\documentclass{article}

\usepackage[english]{babel}

\usepackage[letterpaper,top=2cm,bottom=2cm,left=3cm,right=3cm,marginparwidth=1.75cm]{geometry}

\usepackage[numbers]{natbib}
\usepackage{hyperref}
\usepackage{amsmath}
\usepackage{amssymb}
\usepackage{url}
\usepackage{mathrsfs}
\usepackage{blkarray}
\usepackage{tikz-cd}
\usepackage{amsthm}
\usepackage{mathtools}
\usepackage{enumitem}
\usepackage{wrapfig}

\theoremstyle{plain}
\newtheorem{theorem}{Theorem}
\newtheorem{proposition}{Proposition}[section]
\newtheorem{corollary}[proposition]{Corollary}
\newtheorem{example}[proposition]{Example}
\newtheorem{remark}[proposition]{Remark}
\newtheorem{definition}[proposition]{Definition}
\newtheorem{lemma}[proposition]{Lemma}

\newcommand{\R}{\mathbb{R}}

\newcommand{\define}[1]{\textbf{#1}}

\usepackage{graphicx}

\newcommand{\cev}[1]{\reflectbox{\ensuremath{\vec{\reflectbox{\ensuremath{#1}}}}}}

\title{Stability of Hypergraph Invariants and Transformations}
\author{Tom Needham and Ethan Semrad}
\date{}

\begin{document}
\maketitle

\begin{abstract}
Graphs are fundamental tools for modeling pairwise interactions in complex systems. However, many real-world systems involve multi-way interactions that cannot be fully captured by standard graphs. Hypergraphs, which generalize graphs by allowing edges to connect any number of vertices, offer a more expressive framework. In this paper, we introduce a new metric on the space of hypergraphs, inspired by the Gromov-Hausdorff distance for metric spaces. We establish Lipschitz properties of common hypergraph transformations, which send hypergraphs to graphs, including a novel graphification method with ties to single linkage hierarchical clustering. Additionally, we derive lower bounds for the hypergraph distance via invariants coming from basic summary statistics and from topological data analysis techniques. Finally, we explore stability properties of cost functions in the context of optimal transport. Our results in this direction consider Lipschitzness of the Hausdorff map and conservation of the non-negative cross curvature property under limits of cost functions. 
\end{abstract}

\section{Introduction}

\begin{wrapfigure}{r}{0.45\textwidth} 
\vspace{-20pt}
  \begin{center}
    \includegraphics[width=0.39\textwidth]{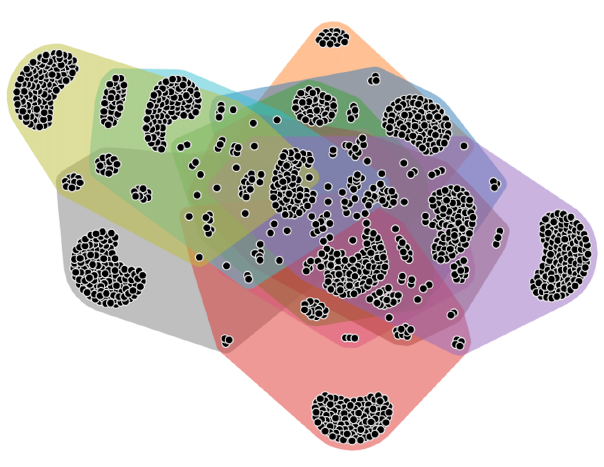}%{./Pictures/mainscreen1.png}
    \vspace{-1pt}
    \caption{A hypergraph representing a gene relation dataset, from \cite{Zhou_2021}. Vertices (visualized as black points) represent genes and edges (visualized as multicolored regions containing the vertices) consist of pathways from the Hallmarks collection within the Molecular Signatures Database. Observe that an edge can contain many more than two nodes, in contrast to the structure of a classical graph.}
    \label{fig:bio4}
  \end{center}
  \vspace{-20pt}
  \vspace{1pt}
\end{wrapfigure} 

Graphs are the canonical formalism for representing data which involves interactions, and are ubiquitous in applications to ecology \cite{Fortin_2021}, gene regulation \cite{Schlitt_2007}, protein flexibility \cite{Jacobs_2001}, social network analysis \cite{Christakis_2007}, and many other areas. Due to their intrinsic structure, graphs are only able to capture pairwise interactions between objects, whereas it is frequently natural to consider multi-way interactions when modeling a complex system. A \emph{hypergraph} is a graph-like structure which more generally allows any number of vertices to be related by an edge. Hypergraphs can exhibit the multiple types of connective possibilities required to capture, e.g., intricate interactions between the trophic levels of an ecosystem or papers in a collaboration network, which would be lost in a more basic graph representation~\cite{Estrada_2006}. Figure \ref{fig:bio4} shows an example of a hypergraph from a gene relation data set.

This paper takes a metric geometry perspective on certain problems in graph and hypergraph theory. Due to their ubiquity in both pure and applied mathematics, there is a need for methods to compare graph structures (in particular, those defined over distinct node sets); one approach which has seen strong development over the last few decades is to compare them through Gromov-Hausdorff-like constructions. The \emph{Gromov-Hausdorff distance} famously provides a metric on the space of compact metric spaces~\cite{gromov1981groups}, and has more recently been extended to compare (generalized) graph structures~\cite{Chowdhury2018,chowdhury2022distances}. Details are provided below in Section~\ref{sec:gromov_haus}, but the general mechanism of these graph distances is based on finding an alignment between the nodes of the graphs under comparison which minimizes a certain distortion function. 

In this article, we extend the Gromov-Hausdorff distance between graphs to a new metric on the space of hypergraphs. Because hypergraphs have edges that encode relations between nodes which are not strictly pairwise, comparison between hypergraphs requires an alignment between both the nodes and the edges; the structure of Gromov-Hausdorff distance leads to a natural extension of this form, as we provide in Definition \ref{def:hypernetwork_distance}. While this new metric has potential applications to real-world hypergraph data analysis, the focus of the present article is on theoretical properties of the metric. In particular, our contributions focus on stability properties of the metric---i.e., demonstrating that certain hypergraph invariants define Lipschitz maps into relatively simple representation spaces. 

\subsection{Main Results and Outline} 

We now outline the paper and describe our main contributions: 

    \smallskip
    \noindent {\bf Section \ref{sec:hypernetworks} (Hypernetworks and Hypernetwork Distances).} A general mathematical model for the notion of a hypergraph, called a \emph{hypernetwork} is introduced in Definition \ref{def:hypernetwork}. This definition is in line with a general model for graph data, called a \emph{network}, studied in \cite{Chowdhury2018,chowdhury2022distances} (while the term \emph{network} is frequently used interchangeably with the term \emph{graph}, the meaning here is specifically as in Definition \ref{def:network_GH}). Hypergraphs form an important class of examples of hypernetworks, but our definition also encodes important mathematical structures such as data matrices or cost functions (see Example \ref{ex:hypernetworks}). Our new distance for hypernetworks, denoted $d_\mathcal{H}$, is introduced in Definition \ref{def:hypernetwork_distance} and is shown to be a metric (up to a natural notion of equivalence) in Theorem \ref{thm:metric}. Our definition is based on a Gromov-Hausdorff-like distance on the space of networks studied in~\cite{Chowdhury2018,chowdhury2022distances}, denoted $d_{\mathcal{N}}$.  
    
    \smallskip
    \noindent {\bf Section \ref{sec:graphification} (Graphification).} A common method in hypergraph analysis is to transform a hypergraph it into a graph via one of several apparently ad hoc methods. We provide theoretical context for these transformations, which we call \emph{graphifications}, by showing that the most commonly used transformations (bipartite graph representations, clique expansions, and line graph representations) are Lipschitz maps from the space of hypernetworks to the space of networks, with respect to $d_\mathcal{H}$ and $d_\mathcal{N}$. These results comprise Theorem \ref{thm:bipartite-biLip} and Corollary \ref{cor:lipschitz}. The latter result is based on an analysis of a novel graphification---the \emph{affinity graph map} (see Definition \ref{def:affinitygraph})---which is shown to be Lipschitz in Theorem \ref{thm:affinitylipschitz}. 
    
    \smallskip
    \noindent {\bf Section \ref{sec:lower_bounds} (Lower Bounds).} We consider several invariants of hypergraphs and use them to give efficiently computable estimates of the hypergraph distance $d_\mathcal{H}$. The first group of invariants are based on simple summary statistics of the hypergraph structure; these invariants are valued in the reals or in the power set of the reals. Extending existing results from \cite{Memoli2012some,Chowdhury_2019} for metric spaces and networks, we derive lower bounds on $d_\mathcal{H}$ in terms of distances between these invariants in Theorem \ref{thm:lowerbounds}.  Next, we establish lower bounds on $d_\mathcal{H}$ based on invariants coming from the Topological Data Analysis literature. Namely, we show that hypergraphs are naturally summarized by persistent homology of their \emph{Dowker complexes} (see \cite{Chowdhury2018}, or Definition \ref{def:dowker_filtrations}), and that the interleaving distance between these summaries gives an estimate of $d_\mathcal{H}$. There are two interesting interpretations of these results: first, they show that the hypergraph invariants are \emph{stable}, in the sense that hypernetworks which are close with respect to $d_\mathcal{H}$ yield similar invariants; second, they provide a method for tractably estimating the hypergraph distance.
    
    \smallskip
    \noindent {\bf Section \ref{sec:stability_of_cost_functions} (Stability of Cost Functions).} In the final section of the paper, we shift perspective and focus on the interpretation of our hypernetwork model as representing cost functions, e.g., in the context of optimal transport theory~\cite{villani2009optimal}. The first main result is related to recent work in \cite{mikhailov2018hausdorff}, which shows that the map which takes a metric space to its space of closed, bounded subsets, endowed with Hausdorff distance, is a Lipschitz map with respect to the Gromov-Hausdorff distance. We first extend this result to the setting of $d_\mathcal{N}$ on the space of networks, which involves a novel proof strategy (Theorem \ref{thm:network_hausdorff_map}); this proof strategy then extends to prove a similar result for $d_\mathcal{H}$ (Theorem \ref{thm:hypernetwork_hausdorff_map}). Intuitively, this result says that if two cost functions are close, then their resulting Wasserstein spaces are close---see Remark \ref{rem:intuitive_interpretation_hausdorff}. Finally, we extend recent work of \cite{leger2024nonnegative} concerning \emph{non-negative cross curvature}, a property of a cost function which arises in the Ma-Trudinger-Wang regularity theory of optimal transport \cite{ma2005regularity}. We show in Theorem \ref{thm:nonnegative_cross_curvature} that non-negative cross curvature of cost functions is preserved under limits in the  $d_\mathcal{H}$ topology.

\subsection{Related Work}

Besides the references to related work and results pointed out above, we would like to compare the approach of the present article to that of the authors (and their collaborators) in~\cite{Chowdhury_2023}. That paper also considers a metric on a general model for hypergraph-like objects; the main distinction therein is that nodes and edges of hypergraphs are endowed with probability measures, so that methods from optimal transport theory become applicable. Specifically, the constructions there borrowed ideas from M\'{e}moli's \emph{Gromov-Wasserstein distance}~\cite{memoli2011gromov,Chowdhury_2019} and from the recently introduced \emph{co-optimal transport} framework~\cite{Redko_2020}. Lipschitzness of graphifications was also considered in \cite{Chowdhury_2023}, but the proof techniques used here are novel (and we believe, in fact, could be adapted to provide Lipschitz constants which are improved over those appearing in \cite{Chowdhury_2023}). Besides, this article works with a variant of Gromov-Hausdorff (rather than Gromov-Wasserstein) distance, which is arguably a more theoretically fundamental construction. 

We end the introduction with a note on exposition style. As some of our results extend those of \cite{Memoli2012some,Chowdhury2018,mikhailov2018hausdorff,chowdhury2022distances,leger2024nonnegative} to the setting of hypernetworks, we have generally aimed for a streamlined presentation of our contributions. All of our results on hypernetworks are novel, but when their proofs follow by minor adaptations of existing proofs in the context of metric spaces or networks, we omit or sketch proofs and point the reader to the relevant results in the literature---this is the case, for example, for Proposition \ref{prop:GH}, Theorem \ref{thm:lowerbounds} and Theorem \ref{thm:dowker}. On the other hand, several of our results are completely new, or require different proof techniques than those that have appeared previously, in which case full details are provided---see, e.g., Theorem \ref{thm:affinitylipschitz} and Theorem \ref{thm:network_hausdorff_map}.  

% ------------------------------------------------------------------------

\section{Hypernetworks and Hypernetwork Distances}\label{sec:hypernetworks}

This section introduces the main structures that will be studied throughout the article. First, we define \emph{(weighted) hypernetworks} as metric space-like structures which give a natural and far-reaching generalization of the notion of a hypergraph. Second, we extend the construction of \emph{Gromov-Hausdorff distance} to define a new distance between hypernetwork structures. We begin by recalling some relevant background.

\subsection{Network Gromov-Hausdorff Distance}\label{sec:gromov_haus}

The \emph{Gromov-Hausdorff (GH) distance}, first studied by Edwards~\cite{edwards1975structure} and later rediscovered and popularized by Gromov~\cite{gromov1981groups}, is a fundamental tool in metric geometry. We refer to \cite{burago2001course} as a standard reference on its basic properties, and to \cite{Memoli2012some}, where further important  properties are established, such as certain lower bounds based on metric invariants. 

The GH distance was initially conceived as a metric on the space of isomorphism classes of compact metric spaces, and has several natural formulations (see \cite{burago2001course}). It was observed in work of Chowdhury and M\'{e}moli that certain formulations of GH distance are amenable to being extended to measure distances between more general objects consisting of sets $X$ endowed with real-valued kernels $X \times X \to \R$ (which do not necessarily satisfy any of the metric axioms) \cite{chowdhury2016distances,Chowdhury2018,chowdhury2022distances}. We recall the main definitions below, and generally follow the terminology and notation of \cite{chowdhury2022distances} throughout the paper.

\begin{definition}[Networks and Network Gromov Hausdorff Distance]\label{def:network_GH}
A \define{network} is a pair $N=(X,\omega)$, where $X$ is a set and $\omega:X\times X\rightarrow \mathbb{R}$ is an arbitrary  function. We use $\mathcal{N}$ to refer to the space of all networks.

Recall that a  \define{correspondence} between sets $X$ and $X'$ is a relation $R \subset X \times X'$ such that for all $x \in X$, there exists $x' \in X'$  such that $(x,x') \in R$, and, likewise, for every $y' \in X'$ there exists $y \in X$ such that $(y,y') \in R$. We use $\mathcal{R}(X,X')$ to denote the set of all correspondences between $X$ and $X'$. 

Let $N=(X,\omega), N'=(X',\omega') \in \mathcal{N}$. The associated \define{network distortion} of a correspondence $R \in \mathcal{R}(X,X')$ is defined as 
\[
\mathrm{dis}_\mathcal{N}(R) \coloneqq \sup_{(x,x'),(y,y') \in R} \vert\omega(x,y)-\omega'(x',y') \vert.
\]
The \define{Network Gromov-Hausdorff Distance} is 
\begin{equation}\label{eqn:GH_network}
    d_{\mathcal{N}}(X,X')\coloneqq \frac{1}{2} \inf_{R\in \mathcal{R}(X,X')} \mathrm{dis}_\mathcal{N}(R).
\end{equation}
\end{definition}

\begin{example}[Graphs as Networks]
    The network structure described above encompasses that of a metric space, but is  much more general. The ``network" terminology is intended to evoke the situation where $X$ is a finite set and $\omega$ is a graph kernel, such as an adjacency function, weighted adjacency function, graph Laplacian, et cetera. Finite networks provide the primary motivation for many of the constructions and results in this paper.
\end{example}

\begin{remark}[Technical Assumptions About Networks]\label{rmk:technical_assumptions}
    In \cite{chowdhury2022distances}, networks were assumed to satisfy additional technical assumptions. In particular, the underlying set of a network $(X,\omega)$ was assumed to be additionally endowed with a first countable topology, and the function $\omega$ was assumed to be continuous with respect to its product topology. We have found these assumptions to be unnecessary for our purposes. In fact, it is explained in \cite[Remark 2.3.8]{chowdhury2022distances} that several of the results of that paper go through without the additional structure assumed on elements of $\mathcal{N}$. These structures were put in place in the previous work for topological reasons which are not pertinent to the present article. The reader should take note that $\mathcal{N}$ is used in \cite{chowdhury2022distances} to denote the class of networks satisfying these additional hypotheses.  
\end{remark}

With a view toward understanding the metric structure of $d_\mathcal{N}$---in particular, the extent to which it fails the positive definiteness axiom---\cite{chowdhury2022distances} introduces a certain equivalence relation on $\mathcal{N}$. We recall the definition below.

\begin{definition}[Weak Isomorphism]\label{def:network_weakiso}
Let $N = (X,\omega), N' = (X',\omega')\in \mathcal{N}$. We say that $N$ and $N'$ are \define{weakly isomorphic}, if for any $\epsilon>0$ there exists a set $Z_\epsilon$ with surjective maps $\varphi_\epsilon:Z_\epsilon \rightarrow X$, $\varphi'_\epsilon:Z_\epsilon \rightarrow X'$, such that
\begin{align}
    \left\vert \omega(\varphi_\epsilon(z),\varphi_\epsilon(z')) - \omega'(\varphi'_\epsilon(z),\varphi'_\epsilon(z'))\right\vert < \epsilon \text{ for each }z,z' \in Z_\epsilon.
\end{align}
We write $N \cong^{w} N'$ to denote two networks as weakly isomorphic.
\end{definition}

\begin{remark}
    In \cite{chowdhury2022distances}, $\cong^{w}$ is referred to as \emph{type-II weak isomorphism}, in order to distinguish it from another equivalence relation. This distinction will not be considered in the present paper, so we have opted for simpler terminology.
\end{remark}

It is shown in \cite{chowdhury2022distances} that $d_\mathcal{N}$ defines a pseudometric on $\mathcal{N}$, with $d_\mathcal{N}(N,N') = 0$ if and only if $N \cong^{w} N'$; that is, $d_\mathcal{N}$ induces a metric on the quotient space $\mathcal{N}/\cong^w$. Indeed, this result is a combination of  Theorem 2.3.7 and Remark 2.3.8 of \cite{chowdhury2022distances}.

\subsection{Weighted Hypernetworks}

The notion of a hypergraph is a generalization of that of a graph: edges in a hypergraph are allowed to join an arbitrary number of vertices, rather than only two. We make this precise, as follows.

\begin{definition}[Hypergraph]\label{def:hypergraph}
A \define{hypergraph} is a pair $(X,Y)$, where $X$ is a finite set  and $Y \subset \mathcal{P}(X)$ is a collection of nonempty subsets of $X$. Elements of $X$ are referred to as \define{nodes} and elements of $Y$ are referred to as \define{hyperedges}. This structure can be represented via a binary \define{incidence function} $\omega$ encoding the containment of the nodes in the hyperedges---that is, $\omega:X \times Y \to \{0,1\}$, with $\omega(x,y) = 1$ if and only if $x \in y$.  

A \define{weighted hypergraph} consists of a hypergraph $(X,Y)$ together with a real number associated to each pair $(x,y) \in X \times Y$ such that $x \in y$. This structure can be encoded as a \define{weighted incidence function} $\omega:X \times Y \to \R$, with $\omega(x,y)$ returning the appropriate weight if $x \in y$ and zero otherwise. 
\end{definition}

\begin{wrapfigure}{r}{0.45\textwidth} 
\vspace{-20pt}
  \begin{center}
    \includegraphics[width=0.39\textwidth]{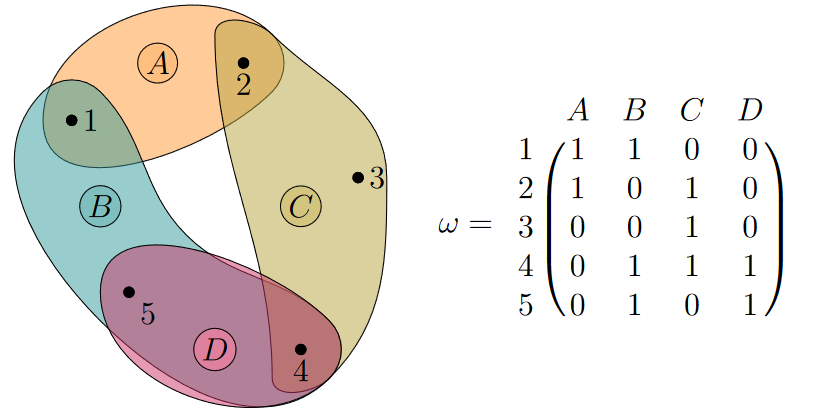}%{./Pictures/mainscreen1.png}
    \vspace{-1pt}
    \caption{A hypergraph with binary incidence function $\omega$. }
    \label{fig:hypergraph}
  \end{center}
  \vspace{-20pt}
  \vspace{1pt}
\end{wrapfigure} 

Hypergraphs are frequently visualized as Venn diagrams, with the node set drawn as a collection of points and with the hyperedges depicted as shaded regions---see Figure \ref{fig:hypergraph} for an example, together with its binary incidence function. Figure \ref{fig:weighted_hyper} depicts a weighted hypergraph with associated weighted incidence function.

By dropping the requirement that $Y \subset \mathcal{P}(X)$, one obtains a flexible model for hypergraphs. Mirroring terminology used in \cite{Chowdhury2018} in the setting of general models for graphs and \cite{Chowdhury_2023} in the setting of models for hypergraphs with additional probabilistic data (so-called \emph{measure hypernetworks}), we define our objects of study as follows.

\begin{definition}[Hypernetwork]\label{def:hypernetwork}
    A \define{hypernetwork} is a triple $H=(X,Y,\omega)$, where $X$ and $Y$ are sets and $\omega:X\times Y\rightarrow \mathbb{R}$ is an arbitrary function. We denote the space of all hypernetworks as $\mathcal{H}$.
\end{definition}

\begin{example}[Examples of Hypernetworks]\label{ex:hypernetworks}
    The hypernetwork structure defined above is rather general and includes a number of objects of interest.
    \begin{enumerate}
        \item {\bf Hypergraphs.} As was described above, a (weighted) hypergraph $(X,Y)$ determines a hypernetwork $(X,Y,\omega)$, where $\omega$ is the (weighted) incidence function. Much of the work in this paper was inspired by thinking of hypernetworks as models for hypergraphs---this perpsective is particularly relevant in Section \ref{sec:graphification}, where we describe theoretical properties of several methods for transforming a hypergraph into a graph. Nonetheless, several other important structures can be modeled as hypernetworks, as we now explain.
        \item {\bf Metric Spaces and Weighted Networks.} The hypernetwork model recovers the notion of a weighted network in a trivial way: a weighted network $(X,\omega)$ can be  represented as a weighted hypernetwork $(X,X,\omega)$.  
        \item {\bf Data Matrices.} A data matrix $M \in \R^{m \times n}$, where each row represents a sample in a dataset and the columns describe features of the samples, determines a hypernetwork $(X,Y,\omega)$, with $X = \{1,\ldots,m\}$, $Y = \{1,\ldots,n\}$ and $\omega(i,j)$ is the $(i,j)$-entry of $M$. The \emph{co-optimal transport} framework of \cite{Redko_2020} was introduced in order to compare data matrices, and this was an inspiration for the measure hypernetwork formalism of \cite{Chowdhury_2023}. The invariants described in Section \ref{sec:lower_bounds} should be broadly useful for comparing data matrices within the hypernetwork framework, due to their theoretical stability properties established below.
        \item {\bf Cost Functions.} Much of the theory of optimal transport~\cite{villani2009optimal} has been developed to compare probability measures $\mu$ on a space $X$ and $\nu$ on a space $Y$, in the presence of some auxiliary cost function $\omega:X \times Y \to \R$. The hypernetwork structure can be interpreted as encoding such costs $H = (X, Y ,\omega)$. In Section \ref{sec:stability_of_cost_functions}, we show that certain operations on these cost structures are \emph{stable}, using the language of hypernetworks. These results should be interesting from the perspective of optimal transport, as we explain below.
    \end{enumerate}
\end{example}

\begin{remark}[Node and Hyperedge Terminology]
    Although hypernetworks model more general objects than hypergraphs (see Example \ref{ex:hypernetworks}), we still sometimes apply hypergraph terminology to hypernetworks. For example, given a hypernetwork $H = (X,Y,\omega)$, we may refer to $X$ as the \define{node set} of $H$ and $Y$ as the \define{edge set} or \define{hyperedge set} of $H$, even though the elements of these sets are not really assumed to be nodes/hyperedges in any precise sense. 
\end{remark}

The notion of weak isomorphism of networks (Definition \ref{def:network_weakiso}) extends naturally to an equivalence relation between  hypernetworks. 

\begin{definition}[Weak Isomorphism of Hypernetworks]
\label{def:weak-iso-1}
Let $H=(X,Y,\omega), H'=(X',Y',\omega')\in \mathcal{H}$. We say $H$ and $H'$ are \define{weakly isomorphic} if for any $\epsilon>0$ there exists sets $W_\epsilon$ and $Z_\epsilon$ with surjective maps $\varphi_\epsilon:W_\epsilon \rightarrow X$, $\varphi'_\epsilon:W_\epsilon \rightarrow X'$, $\psi_\epsilon:Z_\epsilon \rightarrow Y$, and $\psi'_\epsilon:Z_\epsilon \rightarrow Y'$ such that
\begin{align}
    \left\vert\omega(\varphi_\epsilon(w),\psi_\epsilon(z))-\omega'(\varphi'_\epsilon(w),\psi'_\epsilon(z))\right\vert < \epsilon \text{ for each }w\in W_\epsilon, z \in Z_\epsilon.
\end{align}
We recycle notation from Definition \ref{def:network_weakiso} and write $H \cong^{w} H'$ to denote two hypernetworks as weakly isomorphic---the meaning of $\cong^w$ should always be clear from context. 
\end{definition}

\begin{remark}[Weak and Strong Isomorphism]\label{rmk:strong_and_weak}
    The Gromov-Hausdorff distance between two compact metric spaces is zero if and only if they are isometric. An appropriate generalization of isometry to the setting of hypernetworks is as follows. Let $H=(X,Y,\omega)$ and $H'=(X',Y',\omega')$ be hypernetworks. A \define{strong isomorphism} from $H$ to $H'$ is a pair of bijections $\varphi:X \to X'$ and $\psi:Y \to Y'$ such that $\omega(x,y)=\omega'(\varphi(x),\psi(y))$ for all $(x,y) \in X \times Y$. If a strong isomorphism exists, we say that $H$ and $H'$ are \define{strongly isomorphic}.
    
    It is easy to see that if a pair of  hypernetworks are strongly isomorphic, then they are weakly isomorphic. However, the converse does not hold; see Example \ref{ex:distance_zero}. 
\end{remark}

\begin{remark}[The Finite Case]{\label{finite_case_weak_iso}}
    As we mentioned in Example \ref{ex:hypernetworks}, when dealing with hypernetworks, we are frequently interested in the case where all underlying sets are finite. In this case, the weak isomorphism equivalence relation significantly simplifies: if $H,H'\in\mathcal{H}$ have finite underlying sets, then $H\cong^w H'$ if and only if there exists sets $W$ and $Z$ with surjective maps $\varphi:W \rightarrow X$, $\varphi':W \rightarrow X'$, $\psi:Z \rightarrow Y$, and $\psi':Z \rightarrow Y'$ such that
\begin{align*}
\omega(\varphi(w),\psi(z))=\omega'(\varphi'(w),\psi'(z))\text{ for each }w\in W, z \in Z.
\end{align*}
\end{remark}

\subsection{Distance Between Hypernetworks}

We now define a distance between hypernetworks which is structurally similar to the network Gromov-Hausdorff distance $d_\mathcal{N}$ of Definition \ref{def:network_GH}. 

\begin{definition}[Hypernetwork Distance]\label{def:hypernetwork_distance}
Let $H=(X,Y,\omega)$, $H'=(X',Y',\omega') \in \mathcal{H}$. The associated \define{hypernetwork distortion} of a pair of  correspondences $S \in \mathcal{R}(X,X')$ and $T \in \mathcal{R}(Y,Y')$ is defined as
\[ \mathrm{dis}_\mathcal{H}(S,T)\coloneqq\sup_{\substack{(x,x')\in S \\ (y,y')\in T}} \vert\omega(x,y)-\omega'(x',y') \vert.\]
We define the \define{hypernetwork Gromov-Hausdorff distance} (or just \define{hypernetwork distance}) between $H$ and $H'$ to be 
\begin{equation}\label{eqn:GH_corr}
    d_{\mathcal{H}}(H,H')\coloneqq \frac{1}{2} \inf_{\substack{S\in\mathcal{R}(X,X') \\ T\in\mathcal{R}(Y,Y')}} \mathrm{dis}_\mathcal{H}(S,T).
\end{equation}
\end{definition}

We now show that $d_\mathcal{H}$ defines a metric on the space of hypernetworks, considered up to weak isomorphism. This is analogous to the situation for $d_\mathcal{N}$ on the space of networks, proved in \cite[Theorem 2.3.7]{chowdhury2022distances}, and the proof here is very similar. We note that some other results in the paper for $d_\mathcal{H}$ are similar to results known for $d_\mathcal{N}$ and are proved by making small adaptions to the existing proofs. As was mentioned in the introduction, we will omit various proofs which fall into this category; we include the proof of the following result for the sake of completeness and to give an example of how to convert results about networks to results about hypernetworks. 

\begin{theorem}\label{thm:metric}
    The hypernetwork distance $d_\mathcal{H}$ is a pseudometric on $\mathcal{H}$ such that $d_\mathcal{H}(H,H') = 0$ if and only if $H \cong^w H'$. Thus $d_\mathcal{H}$ induces a metric on the quotient space $\mathcal{H}/\cong^w$. 
\end{theorem}

\begin{proof}
    The distance $d_\mathcal{H}$ is clearly positive and symmetric, so to prove that it is a pseudometric, we only need to establish the triangle inequality. Let $H=(X,Y,\omega)$, $H'=(X',Y',\omega')$, $H''=(X'',Y'',\omega'')$ and choose arbitrary  correspondences $S_1\in\mathcal{R}(X,X')$, $S_2\in\mathcal{R}(X',X'')$, $T_1\in\mathcal{R}(Y,Y')$ and $T_2\in\mathcal{R}(Y',Y'')$. Then define $S_2\circ S_1$ as
    \begin{align*}
        S_2 \circ S_1 &= \{ (x,x'')\in X\times X''  \mid  \exists x' \mbox{ such that } (x,x')\in S_1, (x',x'')\in S_2\},
    \end{align*}
    with $T_2 \circ T_1$ defined similarly. It is not hard to show (cf.\ \cite[Exercise 7.3.26]{burago2001course}) that these compositions produce correspondences; that is, $S_2 \circ S_1\in \mathcal{R}(X,X'')$ and $T_2\circ T_1\in \mathcal{R}(Y,Y'')$.
    
    For arbitrary $(x,x'')\in S_2\circ S_1$ and $(y,y'')\in T_2\circ T_1$, let $x'$ be such that $(x,x')\in S_1$, $(x',x'')\in S_2$ and $y'$ be such that $(y,y')\in T_1$, $(y',y'')\in T_2$. Then
    \begin{align*}
        \vert\omega(x,y)-\omega''(x'',y'') \vert 
        &\leq \vert\omega(x,y) -\omega'(x',y')\vert+\vert\omega'(x',y')-\omega''(x'',y'') \vert \\
        &\leq \sup_{\substack{(x,x')\in S_1 \\ (y,y')\in T_1}}\vert\omega(x,y) -\omega'(x',y')\vert+\sup_{\substack{(x',x'')\in S_2 \\ (y',y'')\in T_2}}\vert\omega'(x',y')-\omega''(x'',y'') \vert
        \\
        &= \mathrm{dis}_\mathcal{H}(S_1,T_1) + \mathrm{dis}_\mathcal{H}(S_2,T_2).
    \end{align*}
     As this happens for every value in the above correspondences, we get
     \[\mathrm{dis}_\mathcal{H}(S_1,T_1) + \mathrm{dis}_\mathcal{H}(S_2,T_2) \geq \mathrm{dis}_\mathcal{H}(S_2\circ S_1,T_2 \circ T_1). \]
     As this was for arbitrary correspondences, the infima over correspondences satisfy
    \begin{align*}
    \inf_{\substack{S_1\in\mathcal{R}(X,X') \\ T_1\in\mathcal{R}(Y,Y')}} \mathrm{dis}_\mathcal{H}(S_1,T_1) +  \inf_{\substack{S_2\in\mathcal{R}(X',X'') \\ T_2\in\mathcal{R}(Y',Y'')}} \mathrm{dis}_\mathcal{H}(S_2,T_2)       &\geq  \inf_{\substack{S_1\in\mathcal{R}(X,X') \\ T_1\in\mathcal{R}(Y,Y') \\ S_2\in\mathcal{R}(X',X'') \\ T_2\in\mathcal{R}(Y',Y'')}} \mathrm{dis}_\mathcal{H}(S_2\circ S_1,T_2 \circ T_1) \\
    &\geq \inf_{\substack{S\in\mathcal{R}(X,X'') \\ T\in\mathcal{R}(Y,Y'')}} \mathrm{dis}_\mathcal{H}(S,T).
    \end{align*}
    This then gives us the triangle inequality 
    \[ d_\mathcal{H}(H,H') + d_\mathcal{H}(H',H'') \geq d_\mathcal{H}(H,H'').\]

   It remains to show that for two hypernetworks, $H$ and $H'$, $H \cong^{w} H'$ if and only if $d_\mathcal{H}(H,H')=0$.
    First let us suppose that $H \cong^{w} H'$ and let $\epsilon>0$. Then there exists some $(W_\epsilon,Z_\epsilon)$ with surjective maps $\varphi_\epsilon:W_\epsilon \rightarrow X$, $\psi_\epsilon:Z_\epsilon \rightarrow Y$, $\varphi'_\epsilon:W_\epsilon \rightarrow X'$, and $\psi'_\epsilon:Z_\epsilon \rightarrow Y'$ such that $\left\vert\omega(\varphi(w_\epsilon),\psi_\epsilon(z))-\omega'(\varphi'_\epsilon(w),\psi'(z))\right\vert < \epsilon$ for each $w\in W_\epsilon, z\in Z_\epsilon$. Then we create a node correspondence $S_\epsilon = \{\varphi_\epsilon(w), \varphi'_\epsilon(w) : \forall w\in W_\epsilon\}$ and an edge correspondence $T_\epsilon = \{\psi_\epsilon(z), \psi'_\epsilon(z) : \forall z\in Z_\epsilon\}$. As the maps are surjective, it is clear that both $S_\epsilon$ and $T_\epsilon$ are correspondences and, by construction, 
    \[d_\mathcal{H}(H,H') \leq \sup_{\substack{(x,x')\in S_\epsilon \\ (y,y')\in T_\epsilon}} \vert\omega(x,y)-\omega'(x',y') \vert= \sup_{\substack{w\in W_\epsilon, \\ z\in Z_\epsilon}} \vert\omega(\varphi_\epsilon(w),\psi_\epsilon(z))-\omega'(\varphi'_\epsilon(w),\psi'_\epsilon(z)) \vert <\epsilon\]
    As this holds for any $\epsilon>0$, we have that $d_\mathcal{H}(H,H')=0$.

    Now let us suppose that $d_\mathcal{H}(H,H')=0$. Let $\epsilon>0$, then we have some $S_\epsilon\in\mathcal{R}(X,X')$ and $T_\epsilon\in\mathcal{R}(Y,Y')$ with $\mathrm{dis}_\mathcal{H}(S_\epsilon,T_\epsilon) < \epsilon$. Define $W_\epsilon=S_\epsilon$, $Z_\epsilon=T_\epsilon$, with $\varphi:W_\epsilon \rightarrow X$ and $\psi:Z_\epsilon \rightarrow Y$ defined to be restrictions of the left projection maps, and $\varphi':W_\epsilon \rightarrow X'$, $\psi':Z_\epsilon \rightarrow Y'$ defined to be the right projection maps. For arbitrary $\epsilon$, we have
    \[\sup_{\substack{(x,x')\in S_\epsilon \\ (y,y')\in T_\epsilon}} \vert\omega(\varphi_\epsilon(w),\psi_\epsilon(z))-\omega'(\varphi'_\epsilon(w),\psi'_\epsilon(z)) \vert < \epsilon,\] 
    and consequently $H \cong^{w} H'$.
\end{proof}

\begin{example}[The Necessity of Weak Isomorphism]\label{ex:distance_zero}
    In contrast to the behavior of Gromov-Hausdorff distance between compact metric spaces, two hypernetworks may have distance $0$ even if they are not strongly isomorphic (see Remark \ref{rmk:strong_and_weak}). For example, let $H=(\{x\},\{y\},\omega)$ and $H'=(\{x'_1,x'_2\},\{y'\},\omega')$ with 
    \[ \omega(x,y)=1, \qquad
    \omega'=\begin{blockarray}{cc}
 & y' \\
\begin{block}{c(c)}
  x'_1 & 1 \\
  x'_2 & 1 \\
\end{block}
\end{blockarray}
\]
(the latter being the obvious representation of $\omega'$ as a matrix). If we define the correspondences $S=\{(x,x'_1),(x,x'_2)\}$ and $T=\{(y,y')\}$, we get dis$_\mathcal{H}(S,T)=0$. However these hypernetworks are not strongly isomorphic, as they have different cardinality.
\end{example}

\begin{example}[Hypernetwork GH Lower Bounds Network GH]\label{ex:hypernetwork_lower_bounds_network}
As pointed out in Example \ref{ex:hypernetworks}, any network $N=(X,\omega) \in \mathcal{N}$ can be considered as a hypernetwork $H(N) = (X,X,\omega)$; that is, we have an embedding $\mathcal{N} \rightarrow \mathcal{H}$. Since $d_\mathcal{H}$ involves optimization over a larger set, it is clear that this embedding is Lipschitz, i.e., 
\[
d_\mathcal{H}(H(N),H(N')) \leq d_\mathcal{N}(N,N').
\]
The following example shows that the embedding is not an isometry, in general.

    Consider the networks $N = (\{x_1,x_2\},\omega)$ and $N'=(\{x'_1,x'_2\},\omega')$  with 
    \[ \omega=\begin{blockarray}{ccc}
 & x_1 & x_2 \\
\begin{block}{c(cc)}
  x_1 & 1 & 0 \\
  x_2 & 0 & 1\\
\end{block}
\end{blockarray} \qquad
    \omega'=\begin{blockarray}{ccc}
 & x'_1 & x'_2 \\
\begin{block}{c(cc)}
  x'_1 & 0 & 1 \\
  x'_2 & 1 & 0\\
\end{block}
\end{blockarray}.
\]
As every correspondence between $N$ and $N'$ has distortion equal to 1, $d_\mathcal{N}(N,N')=\frac{1}{2}$. However, we can interpret the function matrices as hypernetworks with $H=(\{x_1,x_2\},\{x_1,x_2\},\omega)$ and $H'=(\{x'_1,x'_2\},\{x'_1,x'_2\},\omega')$. Then the node correspondence $S=\{(x_1,x_1'), (x_2,x_2')\}$ and hyperedge correspondence $T=\{(x_1,x_2'), (x_2,x_1')\}$ will give a distortion of zero so $d_\mathcal{H}(H,H')=0$.
\end{example}

\subsection{Mapping Formulation of Hypernetwork Distance}\label{sec:mapping_formulation}

Here we show that the hypernetwork GH distance can be expressed in terms of distortions of mappings, analogous to a reformulation of the network GH distance derived in  \cite[Proposition 9]{Chowdhury2018} (this reformulation in the setting of metric spaces goes back to at least  \cite{kalton1999distances}). We recall the details for network GH before proceeding: for $N = (X,\omega), N' = (X',\omega') \in \mathcal{N}$, we have
\[
d_\mathcal{N}(N,N') =\frac{1}{2}\inf_{\varphi,\varphi'} \max \{\mathrm{dis}(\varphi),\mathrm{dis}(\varphi'),\mathrm{codis}(\varphi,\varphi'), \mathrm{codis}(\varphi',\varphi)\},
\]
where $\varphi:X \to X'$ and $\varphi':X' \to X$ are functions, and the \define{functional distortions} and \define{codistortions} are defined by
\begin{align*}
&\mathrm{dis}(\varphi) = \sup_{x,y \in X} |\omega(x,y) - \omega'(\varphi(x),\varphi(y))|, \quad \mathrm{codis}(\varphi,\varphi') = \sup_{x \in X, x' \in X'} |\omega(x,\varphi'(x')) - \omega'(\varphi(x),x')| \\
&\mathrm{dis}(\varphi') = \sup_{x',y' \in X'} |\omega(\varphi'(x'),\varphi'(y')) - \omega'(x',y')|, \quad \mathrm{codis}(\varphi',\varphi) = \sup_{x \in X, x' \in X'} |\omega(\varphi(x'),x) - \omega'(x',\varphi(x))|. 
\end{align*}
Note that we abuse and reuse notation here; the meaning of a particular (co)distortion is always clear from context. We extend this reformulation below, abusing notation even more heavily.

\begin{definition}[Mapping GH Distance]
Let $H=(X,Y,\omega)$, $H'=(X',Y',\omega') \in \mathcal{H}$. We define the \define{mapping hypernetwork Gromov-Hausdorff Distance} as 
\begin{equation}\label{eqn:GH_map}
d^\mathrm{map}_{\mathcal{H}}(H,H')\coloneqq\frac{1}{2} \inf_{\varphi,\psi,\varphi',\psi'}  \sup\{\mathrm{dis}(\varphi,\psi),\mathrm{dis}(\varphi',\psi'),\mathrm{codis}(\varphi',\psi),\mathrm{codis}(\varphi,\psi')\},
\end{equation}
where the infimum is over functions $\varphi:X\rightarrow X'$, $\psi:Y\rightarrow Y'$, $\varphi':X'\rightarrow X$, and $\psi':Y'\rightarrow Y$. Recycling the notation above, the \define{functional (co)distortions} are defined as follows: 
\begin{equation}\label{eqn:codist}
\begin{split}
    \mathrm{dis}(\varphi,\psi) &\coloneqq \sup_{\substack{x\in X \\ y\in Y}}\vert \omega(x,y)-\omega'(\varphi(x),\psi(y)) \vert \\
    \mathrm{dis}(\varphi',\psi') &\coloneqq \sup_{\substack{x'\in X' \\ y'\in Y'}}\vert \omega(\varphi'(x'),\psi'(y'))-\omega'(x',y') \vert \\
    \mathrm{codis}(\varphi,\psi') &\coloneqq \sup_{\substack{x\in X \\ y'\in Y'}}\vert \omega(x,\psi'(y'))-\omega'(\varphi(x),y') \vert \\
    \mathrm{codis}(\varphi',\psi) &\coloneqq \sup_{\substack{x'\in X' \\ y\in Y}}\vert \omega(\varphi'(x'),y)-\omega'(x',\psi(y)) \vert.
\end{split}
\end{equation}
\end{definition}

The next result says that the mapping formulation of hypernetwork GH distance is equal to the hypernetwork GH distance, as defined in Definition \ref{def:hypernetwork_distance}. This mimics the situation for the network GH distance, as proved in \cite[Proposition 9]{Chowdhury2018}. The proof here is similar, so we omit it. Our reformulation will be used to relate hypernetwork GH distance to persistent homology below, in Section \ref{sec:persistent_homology}. 

\begin{proposition}\label{prop:GH}
For $H, H' \in \mathcal{H}$, we have that $d_{\mathcal{H}}(H,H')= d^\mathrm{map}_{\mathcal{H}}(H,H')$.
\end{proposition}

\section{Graphification}\label{sec:graphification}

A common technique in hypergraph analysis is to transform a hypergraph into a traditional graph, which has a more tractable structure \cite{Surana_2021}. In this section, we consider several specific \define{graphifications}, or operations which transform a (weighted) hypergraph into a (weighted) graph. The goal of this section is to show that these operations are Lipschitz with respect to $d_\mathcal{H}$ and $d_\mathcal{N}$.

Before defining these operations in Section \ref{sec:graphification_definitions} below, we explain a simplifying assumption adopted throughout this section. Since the operations of interest take the perspective that (hyper)networks are models for (hyper)graphs, which involve finite sets by definition, we will work in the restricted setting where all (hyper)networks have finite underlying sets. We let $\mathcal{FH}$ refer to the \define{space of finite hypernetworks}; that is, those hyper networks $H=(X,Y,\omega)$ where $X$ and $Y$ are finite sets. Similarly, we let $\mathcal{FN}$ denote the \define{space of finite networks}. Besides providing a level of generality which is appropriate for practical applications to (hyper)graphs, this setting offers the additional technical benefits:
\begin{itemize}
    \item All infimums and supremums involved in the definitions of the metrics and graphifications are realized, so that we can instead work with minimimums and maximums.
    \item Weak isomorphisms are easier to work with, as noted in Remark \ref{finite_case_weak_iso}.
\end{itemize}

A consequence of the second point is the following lemma, which gives an interesting characterization of weak isomorphsim that will be used in the proof of Theorem \ref{thm:affinitylipschitz}.

\begin{lemma}\label{lem:weakiso_lift}
    For two hypernetworks, $H=(X,Y,\omega)$ and $H'=(X',Y',\omega')$, there exists hypernetworks $\overline{H}=(\overline{X},\overline{Y},\overline{\omega})$ and $\overline{H'}=(\overline{X},\overline{Y},\overline{\omega}')$, (The same spaces for nodes and hyperedges but with different hypernetwork functions) such that $H$ is weakly isomorphic to $\overline{H}$, $H'$ is weakly isomorphic to $\overline{H'}$ and \begin{align}\label{eqn:weakiso_lift}
        d_\mathcal{H}(H,H')=d_\mathcal{H}(\overline{H},\overline{H'}) = \frac{1}{2}\max_{(x,y)\in (\overline{X},\overline{Y}) } \vert\overline{\omega}(x,y)-\overline{\omega}'(x,y) \vert.
    \end{align}
    We also have a similar version for two networks , $N=(X,\omega)$ and $N'=(X',\omega')$, there exists hypernetworks $\overline{N}=(\overline{X},\overline{\omega})$ and $\overline{N'}=(\overline{X},\overline{\omega}')$ (over the same sets of nodes but with different network functions) such that $N$ is weakly isomorphic to $\overline{N}$, $N'$ is weakly isomorphic to $\overline{N'}$ and \[d_\mathcal{N}(N,N')=d_\mathcal{N}(\overline{N},\overline{N'}) = \frac{1}{2}\max_{x_1,x_2\in \overline{X} }\vert\overline{\omega}(x_1,x_2)-\overline{\omega'}(x_1,x_2) \vert. \]
\end{lemma}
\begin{proof}
    We shall prove the hypernetwork case, as a similar argument will also show the network version. Let $H=(X,Y,\omega)$ and $H'=(X',Y',\omega')$, with $S\in \mathcal{R}(X,X')$ and $T\in \mathcal{R}(Y,Y')$ being a pair of correspondences that give $d_\mathcal{H}(H,H')=0$. Define the node and hyperedge sets as $\overline{X}=\{(x,x'): (x,x')\in S\}$ and $\overline{Y}=\{(y,y'): (y,y')\in T\}$. The we define the standard coordinate projections with $\pi_X:X\times X'\rightarrow X$, $\pi_{X'}:X\times X'\rightarrow X'$, $\pi_Y:Y\times Y'\rightarrow Y$, and $\pi_{Y'}:Y\times Y'\rightarrow Y'$. Then our hypernetwork functions $\overline{\omega}$, and $\overline{\omega'}$ can be defined as 
    \[ \overline{\omega}((x,x'),(y,y')) = \omega(\pi_X(x,x'),\pi_Y(y,y')) \qquad \overline{\omega'}((x,x'),(y,y')) = \omega'(\pi_{X'}(x,x'),\pi_{Y'}(y,y'))\]
    These clearly define surjective maps that respect the original functional structure so we have $H \cong^{w} \overline{H}$ and $H' \cong^{w} \overline{H'}$. So the first part of the equality of (\ref{eqn:weakiso_lift}) follows from weak isomorphism. Then with $\overline{x}$ and $\overline{y}$ representing elements in $\overline{X}$ and $\overline{Y}$ respectively, a change of variables will give us that 
    \[
        d_{\mathcal{H}}(H,H')\coloneqq \frac{1}{2}  \max_{\substack{(x,x')\in S \\ (y,y')\in T}} \vert\omega(x,y)-\omega'(x',y') \vert = \frac{1}{2}\max_{(\overline{x},\overline{y})\in (\overline{X},\overline{Y}) } \vert\overline{\omega}(\overline{x},\overline{y})-\overline{\omega'}(\overline{x},\overline{y}) \vert.
    \]
\end{proof}

\subsection{Definitions of the Graphification Maps}\label{sec:graphification_definitions}

We now define the main graphification maps---that is, maps $\mathcal{FH} \to \mathcal{FN}$---under consideration. In the following definitions, fix $H = (X,Y,\omega) \in \mathcal{FH}$. 

\begin{definition}[Bipartite Graphification]\label{def:bipartite}
We define the \define{bipartite graph} of $H$ to be the network $\mathsf{B}(H) = (X\sqcup Y,\omega_{\mathsf{B}})$, where
\[
\omega_\mathsf{B}(w,z) = \left\{
\begin{array}{cl}
\omega(w,z) & \mbox{if $w \in X$ and $z \in Y$} \\
\omega(z,w) & \mbox{if $z \in X$ and $w \in Y$} \\
0 & \mbox{otherwise.}
\end{array}\right.
\]
The \define{bipartite graphification} is the map $\mathsf{B}:\mathcal{FH} \to \mathcal{FN}$.
\end{definition}

Intuitively, the bipartite graphification represents the hypergraph by defining two node sets---one corresponding to nodes of the original hypergraph, and one corresponding to hyperedges---and connects them in a bipartite fashion in a way which reflects the original hypergraph structure. 

One should observe that the bipartite graphification is invertible, in that the original hypergraph can be reconstructed from its bipartite representation. This is verified below in Section \ref{sec:bipartite}. The following operations do not enjoy the same property. 

\begin{definition}[Clique Expansion and Line Graph Graphifications]\label{def:graphifications}
We define the \define{clique expansion graph} of $H$ to be the network $\mathsf{Q}(H) = (X,\omega_{\mathsf{Q}})$, where
\[
\omega_{\mathsf{Q}}(x_1,x_2) \coloneqq \max_{y\in Y}\left|\min_{j=1,2} \omega(x_j,y) \right|.
\]
The \define{clique expansion graphification} is $\mathsf{Q}:\mathcal{FH} \to \mathcal{FN}$. 
 
Similarly, the \define{line graph} of $H$ is defined to be $\mathsf{L}(H) = (Y,\omega_{\mathsf{L}})$, where
\[
\omega_{\mathsf{L}}(y_1,y_2) \coloneqq \max_{x\in X}\left|\min_{j=1,2} \omega(x,y_j) \right|.
\]
The \define{line graph graphification} is the map $\mathsf{L}:\mathcal{FH} \to \mathcal{FN}$. 
\end{definition}

Intuitively, the line graph $\mathsf{L}(H)$ of $H$ is a graph whose node set corresponds to the set of hyperedges of $H$;  two nodes are adjacent in $\mathsf{L}(H)$ if their corresponding hyperedges have a nonempty intersection in $H$.  The clique expansion $\mathsf{Q}(H)$ is a graph constructed from $H$ by replacing each hyperedge with a clique among its nodes.  With such reduction techniques, there is bound to be some information loss. In fact, it is possible for two structurally different hypernetworks to give you the same line graph or clique graph \cite{kirkland_2017_informationloss}. As such, a natural question is: what information remains after such a reduction technique? We address this question is by establishing Lipschitz bounds for more general in Section \ref{sec:affinity_networks} below.

Simple examples of these graphification operations are illustrated in Figure \ref{fig:graphification}. The behavior of the clique expansion and line graph maps at the level of network functions is shown explicitly in Figure \ref{fig:hypergraph-graph}.

\begin{figure}
	\centering
	\includegraphics[width=1\columnwidth]{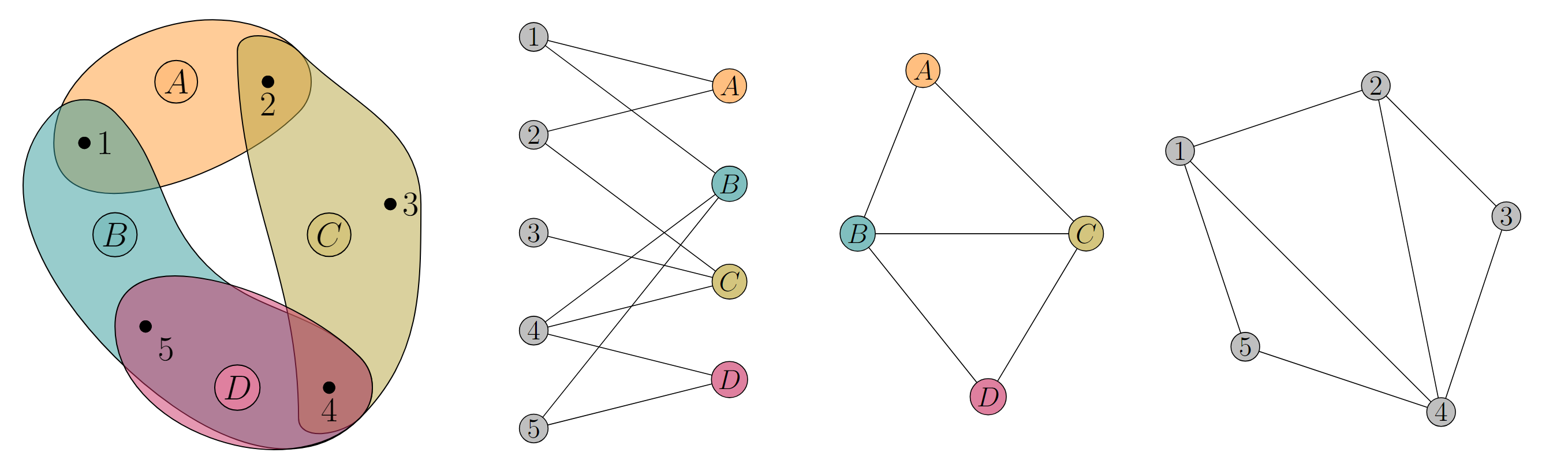}
	\caption{(From left to right) The hypergraph from Figure \ref{fig:hypergraph} with its corresponding bipartite graph, line graph, and clique graph.}	
	\label{fig:graphification}
\end{figure}

\subsection{Bipartite Networks}\label{sec:bipartite}
We now establish theoretical properties of the bipartite graphification (Definition \ref{def:bipartite}). As was informally observed above, the map $\mathsf{B}:\mathcal{FH} \to \mathcal{FN}$ is  invertible. Moreover, it almost respects the metric structures of these spaces---to make this precise, we introduce the following slight variant of the network Gromov-Hausdorff distance. 

\begin{definition}\label{def:labeledbipartite}
    A finite network $N=(Z,\omega)$ is called \define{bipartite} if there exists some partition $Z=X \sqcup Y$ such that for all $x,x'\in X$ and $y,y'\in Y$ we have that $\omega(x,x')=\omega(y,y')=0$ and $\omega(x,y)=\omega(y,x)$. If the bipartite network has a fixed partition $N=(X\sqcup Y,\omega)$ we call it a \define{labeled bipartite network}. We denote the set of labeled bipartite networks as $\mathcal{B}$.

    We say a \define{labeled correspondence} between labeled bipartite networks $N=(X\sqcup Y,\omega)$ and $N=(X\sqcup Y,\omega)$ is a correspondence $R\in \mathcal{R}((X \sqcup Y),(X' \sqcup Y'))$ such that if $(w,z)\in R$ then either $(w,z)\in X\times X'$ or $(w,z)\in Y\times Y'$. We will denote the set of labeled correspondences as $\mathcal{R_B}(X\sqcup Y,X'\sqcup Y')$.

    We define the \define{labeled network distance} between $N,N'\in \mathcal{B}$ as 
\[
    d_{\mathcal{B}}(N,N') \coloneqq \min_{R \in \mathcal{R_B}(X\sqcup Y,X'\sqcup Y')} \mathrm{dis}_\mathcal{N}(R).\]
\end{definition}

We then have the following result.

\begin{theorem}\label{thm:bipartite-biLip}
    The bipartite graph map $\mathsf{B}:\mathcal{FH} \to \mathcal{B}$ is a bijection and $d_{\mathcal{H}}(H,H')=d_{\mathcal{B}}(\mathsf{B}(H),\mathsf{B}(H'))$.
\end{theorem}

A similar result is shown in the probabilistic setting in \cite[Theorem 10]{Chowdhury_2023} (see also \cite{zhang2024geometry}, which generalizes this construction). The proof idea is similar, so we only sketch it here.

\begin{proof}[Proof Sketch]
    The inverse of $\mathsf{B}$ is the map 
    \[ (X \sqcup Y,\omega)\longrightarrow (X,Y,\omega_{\mid_{X\times Y}}), \] 
    where $\omega\mid_{X \times Y}$ is the restriction of $\omega$ to pairs in $X \times Y \subset (X \sqcup Y) \times (X \sqcup Y)$. To show $d_\mathcal{B}(\mathsf{B}(H),\mathsf{B}(H')) \leq d_\mathcal{H}(H,H')$, let $S \in \mathcal{R}(X,Y)$ and $T \in \mathcal{R}(X',Y')$ and define $R_{S,T} = S \sqcup T$. Then $R_{S,T} \in \mathcal{R}_\mathcal{B}(X\sqcup Y,X'\sqcup Y')$ has $d_\mathcal{N}(R_{S,T}) \leq d_\mathcal{H}(S,T)$, and this establishes the desired inequality. The reverse inequality is similar: given a labeled correspondence, one can easily construct correspondences between nodes and hyperedges with smaller hypernetwork distortion.
\end{proof}

\subsection{Affinity Networks}\label{sec:affinity_networks}

\begin{figure}
	\centering
	\includegraphics[width=0.9\columnwidth]{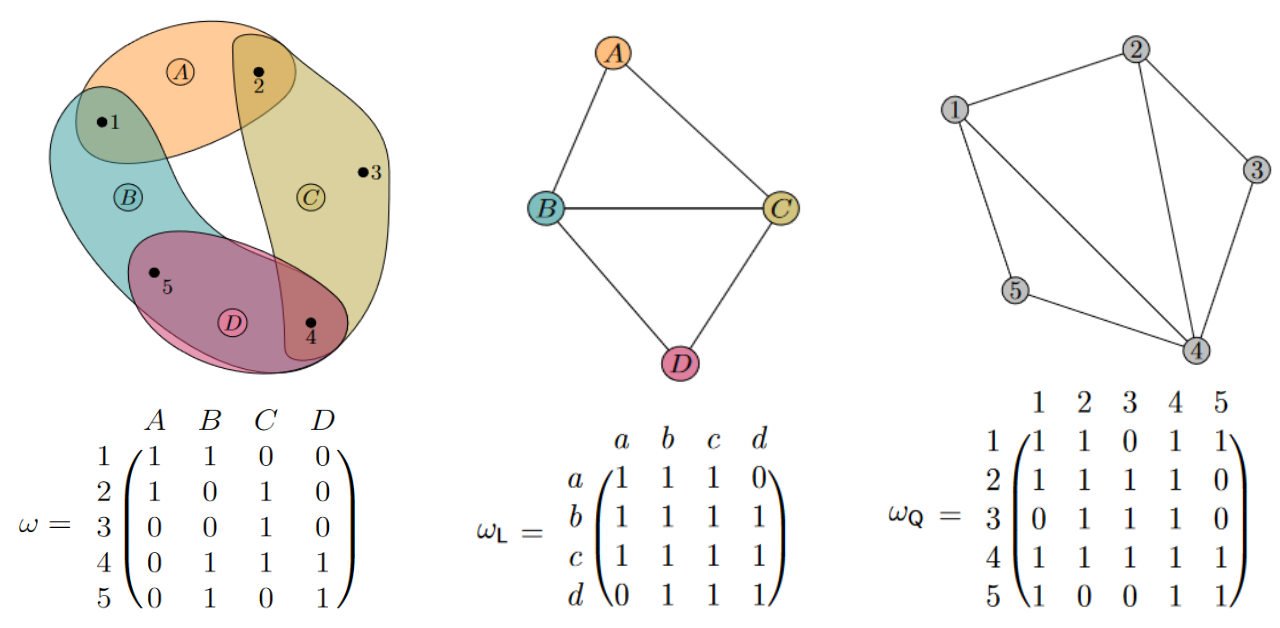}
	\caption{A hypergraph $H$ (left), with its line graph $\mathsf{L}(H)$ (middle), and clique expansion graph $\mathsf{Q}(H)$ (right) and their corresponding omega functions $\omega_{\mathsf{L}}$ and $\omega_{\mathsf{Q}}$. }
	\label{fig:hypergraph-graph}
\end{figure}

The clique expansion and line graph graphifications introduced in Definition \ref{def:graphifications} only examine the local structure between nodes and hyperedges: assuming that $\omega$ encodes a weighted incidence matrix, the clique expansion only connect two nodes if there is some hyperedge that contains both, and a similar statement holds for the line graph. We propose below a new construction, called the \emph{affinity network}, which extends the aforementioned graphifications in a manner which offers a more global understanding of how the nodes (or hyperedges) relate to each other. We begin with some preliminary concepts.

\begin{definition}[Node Chains and Affinity]\label{def:affinity}
    For a finite hypernetwork $H=(X,Y,\omega) \in \mathcal{FH}$, and $x,x' \in X$, a \define{node-chain} from $x$ to $x'$ is a finite sequence $c_n$ of node-edge pairs (i.e., elements of $X \times Y$) with the first node being $x$ and the final node being $x'$: we write
    \[c_n = \big((x_0,y_0),(x_1,y_0),(x_1,y_1),\dots, (x_k,y_k),(x_{k+1},y_k)\big) \quad \mbox{where} \quad x_0=x, x_{k+1}=x', x_i\in X, y_i\in Y. \]
    For $y,y' \in Y$, an \define{edge-chain} $c_e$ between $y$ and $y'$ is defined similarly:
    \[c_e = \big((x_0,y_0),(x_0,y_1),(x_1,y_1),\dots, (x_k,y_k),(x_k,y_{k+1})\big) \quad \mbox{where} \quad y_0=y, y_{k+1}=y', x_i\in X, y_i\in Y. \]
    The collections of all node-chains between $x$ and $x'$ and edge-chains between $y$ and $y'$ will be denoted as $C_n(x,x')$ and $C_e(y,y')$, respectively.
    
    For a chain $c_*$ (either a node-chain $c_n$ or an edge chain $c_e$), we write $(x,y) \in c_*$ to indicate that the node-edge pair $(x,y)$ belongs to the underlying set of $c_*$. The \define{energy} of the chain is defined as
    \[\mathcal{E}_*(c_*)\coloneqq \min_{(x,y)\in c_*}\left|\omega(x,y) \right|.\]
     and the \define{affinity} between two nodes $x,x'$ or between two between two hyperedges $y,y'$ is defined respectively as
     \[\mathcal{A}_n(x,x')\coloneqq \max_{c_n\in C_n(x,x')} \mathcal{E}_n(c_n), \qquad \mathcal{A}_e(y,y')\coloneqq \max_{c_e\in C_e(y,y')} \mathcal{E}_e(c_e) . \]
\end{definition}

We will abuse notation and drop the subscripts for energy and affinity from here on out. The domain of the functionals is clear from context and it will make the following equations less cumbersome. We are now prepared to define a pair of new graphification maps.

\begin{definition}[Node- and Edge-Affinity Networks]\label{def:affinitygraph}
    For a finite hypernetwork $H=(X,Y,\omega) \in \mathcal{FH}$, the \define{node-affinity graph} of $H$ is defined to be the network $\mathsf{A_n}(H)=(X,\omega_\mathsf{A_n})$, where
    \[ \omega_{\mathsf{A_n}}(x_1,x_2) \coloneqq \mathcal{A}(x_1,x_2) \]
    and \define{edge-affinity graph} of $H$ to be the network $\mathsf{A_e}(H)=(Y,\omega_\mathsf{A_e})$, where
    \[ \omega_{\mathsf{A_e}}(y_1,y_2) \coloneqq \mathcal{A}(y_1,y_2) .\]
\end{definition}

\begin{lemma}\label{lem:affinty}
    The node affinity graph and edge affinity graph maps $\mathsf{A_n},\mathsf{A_e}:\mathcal{FH} \to \mathcal{FN}$ respect weak isomorphism, i.e, if $H \cong^{w} H'$, then $\mathsf{A_n}(H) \cong^{w} \mathsf{A_n}(H')$ and $\mathsf{A_e}(H) \cong^{w} \mathsf{A_e}(H')$. 
\end{lemma}

\begin{proof}
    It will suffice to show this for just the node affinity graph map as the edge affinity will follow a similar argument. 
    Now let us suppose that $H \cong^{w} H'$, so $d_\mathcal{H}(H,H')=0$. Therefore we have some $S\in\mathcal{R}(X,X')$ and $T\in\mathcal{R}(Y,Y')$ with zero distortion,
    \[ \mathrm{dis}_\mathcal{H}(S,T)=\max_{\substack{(x,x')\in S \\ (y,y')\in T}} \vert\omega(x,y)-\omega'(x',y') \vert=0.\]
    Then for any $(x,y)\in X \times Y$, via $S$ and $T$ we have a $(x',y')\in X' \times Y'$ such that $\omega(x,y)=\omega'(x',y')$ Now for $x_1,x_2\in X$, consider $C_n(x_1,x_2)$. From the prior statement, for any $(x,y)$ in a node chain $c_n \in C_n(x_1,x_2)$, we can follow $S$ and $T$ to find a corresponding $(x',y')\in H'$. These pairs can be used to generate a chain $c_n' \in C_n(x'_1,x'_2)$ for every $(x_1,x'_1),(x_2,x'_2)\in S$ such that $\mathcal{E}(c_n) =\mathcal{E}(c_n') $. Thus for any $(x_1,x'_1),(x_2,x'_2)\in S$, we have that $\mathcal{A}(x_1,x_2) =\mathcal{A}(x'_1,x'_2) $. Therefore 
       \[ \mathrm{dis}_\mathcal{N}(S)=\max_{(x_1,x'_1),(x_2,x'_2)\in S} \vert\omega_{\mathsf{A_n}}(x_1,x_2) -\omega'_{\mathsf{A_n}}(x'_1,x'_2)  \vert=0.\]
       So $d_N(\mathsf{A_n}(H),\mathsf{A_n}(H') )=0$ and consequently $\mathsf{A_n}(H) \cong^{w} \mathsf{A_n}(H')$. 
\end{proof}

\begin{theorem}\label{thm:affinitylipschitz}
The node affinity graph and edge affinity graph maps $\mathsf{A_n},\mathsf{A_e}:\mathcal{FH} \to \mathcal{FN}$ are $1$-Lipschitz with respect to $d_{\mathcal{H}}$ and $d_{\mathcal{N}}$.
\end{theorem}
\begin{proof}
    Consider two hypernetworks $H=(X,Y,\omega)$ and $H'=(X',Y',\omega')$ and their corresponding node affinity graphs $N=(X,\omega_{\mathsf{A_n}})$ and $N'=(X',\omega'_{\mathsf{A_n}})$.  By Lemma \ref{lem:affinty}, we are allowed to work up to weak isomorphism, and Lemma \ref{lem:weakiso_lift} then implies that we may assume without loss of generality that $X=X'$ and $Y=Y'$, and that 
    \[
    d_\mathcal{H}(H,H') = \frac{1}{2} \max_{(x,y) \in X \times Y} |\omega(x,y) - \omega'(x,y)|. 
    \]
    In the following, we use $\mathcal{E}$ to denote the node-energy function with respect to $\omega$ and we use $\mathcal{E}'$ to denote the function with respect to $\omega'$---that is, $\mathcal{E}'(c_n) = \min_{(x,y) \in c_n} |\omega'(x,y)|$. 
    
    Taking the identity correspondence between $X$ and $X' = X$,we have
    \begin{align}
        2 \cdot d_{\mathcal{N}}(\mathsf{A_n}(H),\mathsf{A_n}(H')) &\leq  \max_{x_1,x_2\in X }\left|\omega_{\mathsf{A_n}}(x_1,x_2) -\omega'_{\mathsf{A_n}}(x_1,x_2)  \right| \nonumber\\
        &= \max_{x_1,x_2\in X }\left|\max_{c_n\in C_n(x_1,x_2)}\mathcal{E}(c_n) -\max_{c_n'\in C_n(x_1,x_2)}\mathcal{E}'(c_n')  \right|
        \nonumber\\
        &\leq  \max_{x_1,x_2\in X }\max_{c_n\in C_n(x_1,x_2)}\left|\mathcal{E}(c_n) -\mathcal{E}'(c_n)  \right|
        \label{eqn:lip_a}\\
        &= \max_{x_1,x_2\in X }\max_{c_n\in C_n(x_1,x_2)} \left| \min_{(x,y) \in c_n} |\omega(x,y)| - \min_{(x',y') \in c_n} |\omega'(x',y')|\right| \nonumber \\
        &= \max_{x_1,x_2\in X }\max_{c_n\in C_n(x_1,x_2)} \min_{\big((x,y),(x',y')\big) \in c_n \times c_n} \left| |\omega(x,y)| - |\omega'(x',y')|\right| \nonumber \\
        &\leq  \max_{x_1,x_2\in X }\max_{c_n\in C_n(x_1,x_2)}\min_{(x,y) \in c_n}\left|\left|\omega(x,y) \right| -\left|\omega'(x,y) \right|  \right|
        \label{eqn:lip_b}\\
        &\leq\max_{x_1,x_2\in X }\max_{c_n\in C_n(x_1,x_2)}\min_{(x,y) \in c_n}\left|\omega(x,y)  -\omega'(x,y) \right|\label{eqn:lip_c}\\
        &\leq\max_{(x,y) \in X \times Y}\left|\omega(x,y)  -\omega'(x,y) \right|\label{eqn:lip_d}\\
        & = 2 \cdot d_\mathcal{H}(H,H') \nonumber.
    \end{align}
     For line (\ref{eqn:lip_a}), we consider $\mathcal{E},\mathcal{E}'$ as a nonnegative real-valued functions on the set $C_n(x_1,x_2)$; then this estimate follows by the reverse triangle inequality applied to the maximum norm on the space of functions $C_n(x_1,x_2) \to \R$. Line (\ref{eqn:lip_b}) comes by restricting the minimum from the previous line to a minimum over the restricted set $\{\big((x,y),(x',y')\big) \in c_n \times c_n \mid x=x', y=y'\}$. Line (\ref{eqn:lip_c}) is an application of the reverse triangle inequality. Then the final line (\ref{eqn:lip_d}) comes from first swtiching the minimum to a maximum, then running the max over a larger set. 

     This proves the claim for the node affinity map. The proof for the edge affinity map follows the same argument.
\end{proof}

\begin{corollary}\label{cor:lipschitz}
The clique expansion graph and line graph maps $\mathsf{Q},\mathsf{L}:\mathcal{FH} \to \mathcal{FN}$ are 1-Lipschitz with respect to $d_{\mathcal{H}}$ and $d_{\mathcal{N}}$. 
\end{corollary}
\begin{proof}
    The clique expansion map is defined in the same way as the node affinity network, but with the restriction that chains must be of length 2. This gives
\[ d_\mathcal{N}(\mathsf{Q}(H),\mathsf{Q}(H')) \leq d_\mathcal{N}(\mathsf{A_n}(H),\mathsf{A_n}(H')) \leq d_\mathcal{H}(H,H').\]
    The argument for the line graph map is the same.
\end{proof}

\begin{remark}
    Corollary \ref{cor:lipschitz} is closely related to \cite[Theorem 11]{Chowdhury_2023}, where a similar Lipchitz bound is established for  versions of the clique expansion and line graph maps which incorporate weights on the node and hyperedge sets (in the form of probability measures). That article does not consider the generalized affinity network construction, and the proof strategy is different than that of Theorem \ref{thm:affinitylipschitz} and Corollary \ref{cor:lipschitz}. Our proofs here give improved Lipschitz bounds, and we suspect that they could be adapted to strengthen the results in the weighted case. 
\end{remark}

\begin{wrapfigure}{r}{0.45\textwidth} 
\vspace{-20pt}
  \begin{center}
\includegraphics[width=.39\textwidth]{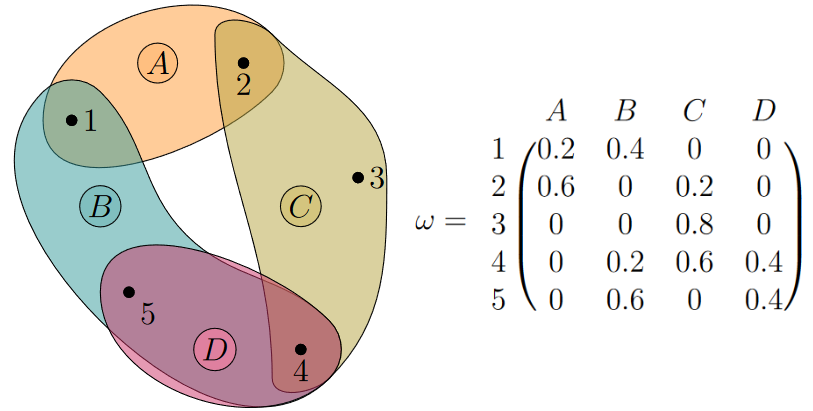}
    \vspace{-1pt}
    \caption{The same hypernetwork used in Figure \ref{fig:hypergraph} except the $\omega$ function is altered to be a weighted function.}
    \label{fig:weighted_hyper}
  \end{center}
  \vspace{-20pt}
  \vspace{1pt}
\end{wrapfigure} 

Here we note a connection between our affinity networks and single linkage hierarchical clustering (SLHC). The input to the SLHC algorithm is a finite metric space $(X,d)$ and the output is a sequence of partitions of $X$, graded by a real number which we refer to as \emph{height}, such that the partitions become coarser at larger heights. At height zero, each element of $X$ forms its own partition block, or \emph{cluster}. Inductively, the distance between clusters $A,B \subset X$ at height $h$ is given by $\min_{a \in A, y \in B} d(a,b)$ and a pair of clusters is merged to form a new partition at height $h' > h$ which realizes the smallest distance between two existing clusters. It is shown in \cite{carlsson2010characterization} that the full multiscale clustering structure of SLHC is captured by the following \emph{ultrametric} structure $u$ on $X$:
\begin{equation}\label{eqn:SLHC_ultrametric}
u(x,x')=\min_{x_0,\ldots,x_k} \max_i d(x_i,x_{i+1}),
\end{equation}
where the minimum is over finite sequences $x_0,x_1,\ldots,x_k \in X$ such that $x_0 = x$ and $x_k = x'$---that this defines an \emph{ultrametric} means that it is a metric which additionally satisfies the strong triangle inequality,
\begin{equation}\label{eqn:strong_triangle_inequality}
u(x,x'') \leq \max\{u(x,x'),u(x',x'')\}.
\end{equation}
There is apparently a similarity in structure between the ultrametric \eqref{eqn:SLHC_ultrametric} associated to a metric space and the affinity networks (Definition \ref{def:affinitygraph}) associated to a hypernetwork. One observation is that the roles of minima and maxima are switched in the two constructions---this is a matter of perspective, as metrics quantify \emph{distance}, whereas we conceptualize hypernetwork kernels as quantifying \emph{affinity}. This similarity is further clarified in the following remark. 

\begin{remark}\label{rem:affinity_relation}
    Let $(X,Y,\omega) \in \mathcal{FH}$. The associated node affinity network $(X,\omega_{\mathsf{A_n}})$ satisfies, for all $x,x',x'' \in X$, 
    \[\omega_{\mathsf{A_n}}(x,x'')\geq \min \{\omega_{\mathsf{A_n}}(x,x'),\omega_{\mathsf{A_n}}(x',x'')\}.\]
    Indeed, for any $c_n \in C_n(x,x')$ and $c_n'\in C_n(x',x'')$, one can concatenate the chains to get $c''_n \in C_n(x,x'')$. Then $\mathcal{E}(c_n'') \geq \min \{\mathcal{E}(c_n),\mathcal{E}(c_n')\}$, and the result easily follows. A similar statement holds for the edge affinity network. 

    In fact, if we reverse the signs of $\omega_{\mathsf{A_n}}$, we get a kernel satisfying the strong triangle inequality (as in \eqref{eqn:strong_triangle_inequality}):
    \[
    -\omega_{\mathsf{A_n}}(x,x'')\leq -\min \{\omega_{\mathsf{A_n}}(x,x'),\omega_{\mathsf{A_n}}(x',x'')\} = \max \{-\omega_{\mathsf{A_n}}(x,x'),-\omega_{\mathsf{A_n}}(x',x'')\}.
    \]
\end{remark}

\begin{figure}[t]
\begin{center}
\includegraphics[width=.39\columnwidth]{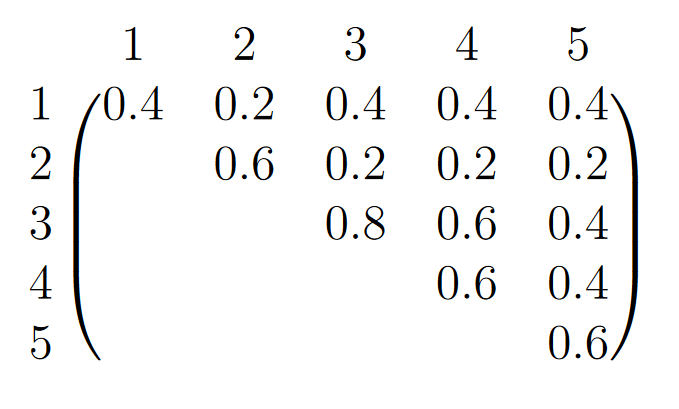} \qquad \qquad 
\includegraphics[width=.32\columnwidth]{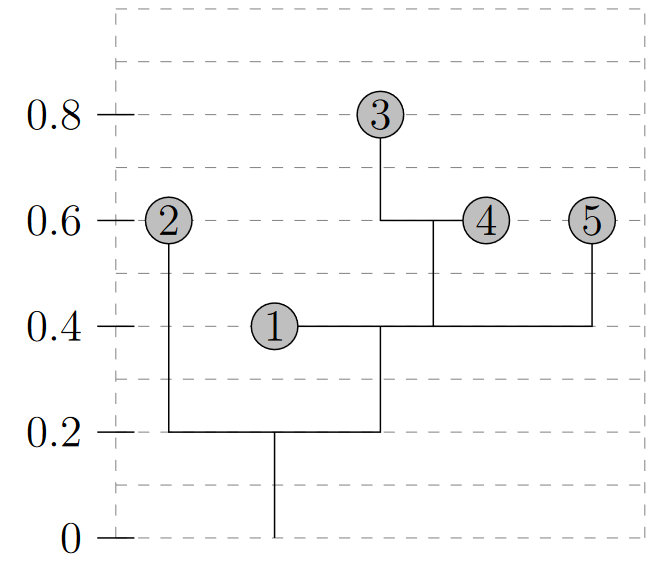}
\caption{The node affinity network created in Example \ref{ex:affinity} with matrix representation (left) and dendrogram-like representation (right).}\label{fig:affinity}
\end{center}
\end{figure}

\begin{example}\label{ex:affinity}
    To illustrate the affinity network construction, we take our example hypernetwork from Figure \ref{fig:hypergraph} and modify the incidence function $\omega$ to give it some variation in weights. The hypernetwork and modified function are shown in Figure \ref{fig:weighted_hyper}. 
    For this hypernetwork, and example node chain between nodes 1 and 5 is
    \[c' = \{(1,A),(2,A),(2,C),(4,C),(4,D),(5,D).\}\]
    The energy of such a chain is
    \[\mathcal{E}(c')=\min |(0.2),(0.6),(0.2),(0.4),(0.4)|=0.2 .\]
    However, the affinity between nodes 1 and 5 is 0.4, being the maximum energy over all chains between them, such a chain can be made by going through hyperedge $B$. If we were to calculate the full node affinity network, the weighted adjacency matrix can be seen in Figure \ref{fig:affinity} (left). We further represent the affinity structure of this hypernetwork via a \emph{dendrogram}-like graph (see \cite{carlsson2010characterization}) in Figure \ref{fig:affinity} (right). The nodes of the network appear on the level of their diagonal entry (as it is their maximum value), then the nodes become clustered at the level when every entry of the subset of nodes is greater than or equal to that value.
\end{example}

\section{Lower Bounds}\label{sec:lower_bounds}

This section introduces several computable lower bounds on the hypernetwork Gromov-Hausdorff distance, based on similar bounds appearing in the context of classical Gromov-Hausdorff distance \cite{Memoli2012some,chowdhury2022distances}.

\subsection{Basic Invariants}\label{subsec:invariants}
The first collection of lower bounds for hypernetwork GH distance involves several invariants of finite hypernetworks.

\begin{definition}[Hypernetwork Invariants]\label{def:hypernetwork_invariants}
    Let $H=(X,Y,\omega) \in \mathcal{FH}$ be a finite hypernetwork. We define the following invariants of $H$:
    \begin{enumerate}
    \item \define{Hypernetwork Capacity}: $\mathbf{Cap}(H)\coloneqq \max_{x\in X, y\in Y}\omega(x,y)$ 
    \item \define{Node-Capacity Function}: $\mathbf{Cap}^n_H:X\rightarrow \mathbb{R}$, with $\mathbf{Cap}^n_H(x)\coloneqq \max_{y\in Y}\omega(x,y)$
    \item \define{Edge-Capacity Function}: $\mathbf{Cap}^e_H:Y\rightarrow \mathbb{R}$, with $\mathbf{Cap}^e_H(y) \coloneqq \max_{x\in X}\omega(x,y)$
    \item \define{Node-Circum-Radius}: $\mathbf{Rad}^n(H)\coloneqq \min_{x\in X}\max_{y\in Y}\omega(x,y)$
    \item \define{Edge-Circum-Radius}: $\mathbf{Rad}^e(H)\coloneqq \min_{y\in Y}\max_{x\in X}\omega(x,y)$
    \item \define{Hypernetwork Spectrum}: $\mathbf{Spec}(H)\coloneqq \{\omega(x,y)\vert x\in X, y\in Y\}$
    \item \define{Node Spectrum Function}: $\mathbf{Spec}^n_H:X\rightarrow \mathcal{P}(\mathbb{R})$, with $\mathbf{Spec}^n_H(x) \coloneqq \{\omega(x,y)\vert  y\in Y\}$
    \item \define{Edge Spectrum Function}: $\mathbf{Spec}^e_H:Y\rightarrow \mathcal{P}(\mathbb{R})$, with $\mathbf{Spec}^e_H(y) \coloneqq \{\omega(x,y)\vert  x\in X\}.$
\end{enumerate}
\end{definition}

In the above, $\mathcal{P}(\R)$ denotes the power set of $\R$. These are \emph{invariants} in the sense that they are invariant under weak isomorphism, as we explain precisely in Corollary \ref{cor:invariant_under_weak_iso}. 

Recall (see Example \ref{ex:hypernetwork_lower_bounds_network}) that a finite network $N = (X,\omega) \in \mathcal{FN}$ can be represented as a hypernetwork $H(N) = (X,X,\omega)$. We can extend the hypernetwork invariants above to network invariants by composing the invariants with the map $N \mapsto H(N)$; for example,
\[
\mathbf{Cap}(N) \coloneqq \mathbf{Cap}(H(N)).
\]
The resulting network invariants correspond to those previously considered in \cite[Section 3]{Memoli2012some} (in the context of metric spaces) and \cite[Section 4]{chowdhury2022distances} (in the context of networks). 

The main theorem for this subsection shows that these invariants are stable with respect to the hypernetwork Gromov-Hausdorff distance. In order to state it precisely, we recall that the \define{Hausdorff distance} between a pair of subsets $A$ and $B$ of a metric space $(Z,d)$ is given by 
\begin{equation}\label{eqn:hausdorff_distance}
d_{\mathrm{Haus}}^Z(A,B)\coloneqq\max\left\{ \sup_{a\in A}\inf_{b\in B} d(a,b), \sup_{b\in B}\inf_{a\in A} d(a,b) \right\}.
\end{equation}
It is easy to see that the Hausdorff distance (between finite sets) is polynomial-time computable, so that the lower bounds in the following theorem give tractable estimates of the hypernetwork GH distance.

\begin{theorem}\label{thm:lowerbounds}
    Let $H=(X,Y,\omega)$, $H'=(X',Y',\omega')$ be finite hypernetworks and define $F_n:X\times X'\rightarrow \mathbb{R}^+$, $F_e:Y\times Y'\rightarrow \mathbb{R}^+$ by
    \[
    F_n(x,x')= \min_{T \in \mathcal{R}(Y,Y')} \max_{(y,y')\in T}\left|\omega(x,y)-\omega'(x',y')\right|, 
    \; F_e(y,y')= \min_{S \in \mathcal{R}(X,X')} \max_{(x,x')\in S}\left|\omega(x,y)-\omega'(x',y')\right|.
    \]
    We have the following bounds in terms of node-based invariants:
    \begin{align}
        d_{\mathcal{H}}(H,H') &= \frac{1}{2} \min_{S \in \mathcal{R}(X,X')} \max_{(x,x')\in S} F_n(x,x')\label{eqn:local} \\
        & \geq \frac{1}{2}\min_{S \in \mathcal{R}(X,X')} \max_{(x,x')\in S}d_{\mathrm{Haus}}^\mathbb{R}(\mathbf{Spec}^n_{H}(x) , \mathbf{Spec}^n_{H'}(x'))\label{eqn:localdist} 
        \\ &=: \mathcal{L}_n(H,H'), \notag
    \end{align}
    with 
    \begin{align}
        \mathcal{L}_n(H,H') &\geq \frac{1}{2} \min_S \max_{(x,x')\in S}\left|\mathbf{Cap}^n_{H}(x)-\mathbf{Cap}^n_{H'}(x') \right| \label{eqn:ecc1} \\
        &\geq \frac{1}{2} d_{\mathrm{Haus}}^\mathbb{R}\left(\mathbf{Cap}^n_{H}(X),\mathbf{Cap}^n_{H'}(X') \right) \label{eqn:ecc2} \\
        &\geq \frac{1}{2} \max \left( |\mathbf{Cap}(H)-\mathbf{Cap}(H')|,|\mathbf{Rad}^n(H)-\mathbf{Rad}^n(H')| \right) \label{eqn:ecc3}
    \end{align}
    and
    \begin{align}
        \mathcal{L}_n(H,H') &\geq  \frac{1}{2} d_{\mathrm{Haus}}^\mathbb{R}\left(\mathbf{Spec}(H),\mathbf{Spec}(H') \right) \label{eqn:diam1} \\
        &\geq \frac{1}{2} |\mathbf{Cap}(H)-\mathbf{Cap}(H')|. \label{eqn:diam2}
    \end{align}
    Similar lower bounds hold in terms of edge-based invariants.
    Moreover, all hypernetwork invariants in the lower bounds can be replaced with network invariants of clique expansion, line graph or affinity network graphifications; for example,
    \[
    d_\mathcal{H}(H,H') \geq \frac{1}{2} \vert \mathbf{Cap}(\mathsf{A_n}(H)) - \mathbf{Cap}(\mathsf{A_n}(H')) \vert.
    \]
    \end{theorem}

\begin{proof}
    The proofs of the inequalities follow those of \cite[Theorem 3.4]{Memoli2012some} (which covers the case of compact metric spaces) and \cite[Propositions 4.3.2 and 4.3.4]{chowdhury2022distances} (covering the case of networks), where one only needs to check that the ideas still work when a pair of correspondences is used, rather than a single correspondence. The proof of the last statement follows by applying the aforementioned results on network invariants and the Lipschitz bounds from Theorem \ref{thm:affinitylipschitz} and Corollary \ref{cor:lipschitz}. 
\end{proof}

We have the following immediate corollary.

\begin{corollary}\label{cor:invariant_under_weak_iso}
    Each of the hypernetwork invariants from Definition \ref{def:hypernetwork_invariants} are  invariant under weak isomorphism, in the sense that the methods of comparing them introduced in Theorem \ref{thm:lowerbounds} vanish when the hypernetworks are weakly isomorphic.    
\end{corollary}

\begin{example}\label{ex:lowerbounds}
    Let us consider the two hypernetworks whose weighted functions can be seen in Figure \ref{fig:weightedomegas}. The various invariant measurements for these hypernetworks are recorded in Table \ref{tab:lowerbounds}. 
    
    Equations \eqref{eqn:diam2} and \eqref{eqn:diam1} give us a lower bound of 0.05. Additionally, when considering the node variants, Equations \eqref{eqn:ecc1}, \eqref{eqn:ecc2}, and \eqref{eqn:ecc3} also give us the same lower bound of 0.05. However, if we consider the edge variants of the latter three, we get a lower bound of 0.15 for all three.

    \begin{table}[b]
        \centering
        \caption{Various Invariant Measurements for the hypernetworks in Example \ref{ex:lowerbounds}}\label{tab:lowerbounds}
        \begin{tabular}{|l|l|l|}
        \hline
                          & $H$                 & $H'$                \\ \hline
        Hypernetwork Capacity     & 0.8                 & 0.7                 \\ \hline
        Node-Circum-Radius        & 0.4                 & 0.5                 \\ \hline
        Edge-Circum-Radius        & 0.4                 & 0.7                 \\ \hline
        Node-Capacity Function    & \{0.4,0.6,0.8\}     & \{0.5,0.7\}         \\ \hline
        Edge-Capacity Function    & \{0.4,0.6,0.8\}     & \{0.7\}             \\ \hline
        Hypernetwork Spectrum Set & \{0.2,0.4,0.6,0.8\} & \{0.1,0.3,0.5,0.7\} \\ \hline
        \end{tabular}
    \end{table}
\end{example}

\begin{figure}[t]
\begin{center}
\includegraphics[width=0.25\columnwidth]{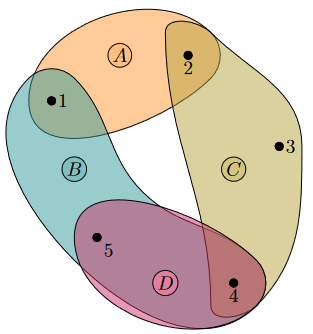} 
\qquad \qquad
\includegraphics[width=0.25\columnwidth]{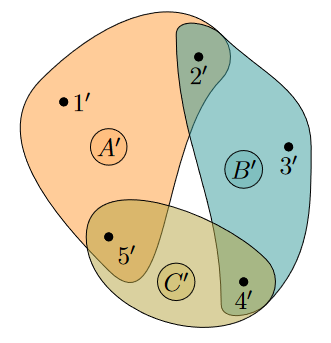}
\[ 
\omega=\begin{blockarray}{ccccc}
 & A & B & C & D \\
\begin{block}{c(cccc)}
  1 & 0.2 & 0.4 & 0 & 0  \\
  2 & 0.6 & 0 & 0.2 & 0  \\
  3 & 0 & 0 & 0.8 & 0  \\
  4 & 0 & 0.2 & 0.6 & 0.4  \\
  5 & 0 & 0.6 & 0 & 0.4  \\
\end{block}
\end{blockarray}
\qquad  \qquad 
\omega'=\begin{blockarray}{cccc}
 & A' & B' & C' \\
\begin{block}{c(ccc)}
  1' & 0.7 & 0 & 0  \\
  2' & 0.5 & 0.3 & 0  \\
  3' & 0 & 0.7 & 0  \\
  4' & 0 & 0.1 & 0.7  \\
  5' & 0.3 & 0 & 0.5  \\
\end{block}
\end{blockarray} \qquad
\]
\caption{The weighted incidence functions for hypernetworks $H$ (left) and $H'$ (right) in Example \ref{ex:lowerbounds}.}\label{fig:weightedomegas}
\end{center}
\end{figure}

\subsection{Persistent Homology}\label{sec:persistent_homology}

Persistent homology is a method of Topological Data Analysis (TDA) that is particularly useful for understanding the underlying topology of data sets and has a large number of applications in areas such as computational biology, image classification, network analysis, among many others; see the survey articles \cite{Edelsbrunner_2008_PH,Fugacci_2016_PH,Aktas_2019_PH} for further references to these applications. In this subsection, we assume that the reader is familiar with the basic constructions of TDA and persistent homology, as are described in, for example, the monograph~\cite{dey2022computational}. 

Most commonly, the persistent homology method applies to a one-parameter family of simplicial complexes with simplicial maps between them---this structure is referred to as a \define{filtered simplicial complex}. We now describe a filtered simplicial complex which is naturally associated to any finite hypergraph. The \emph{Dowker filtrations} seen below are extensions of the \emph{network Dowker filtrations} seen in \cite{Chowdhury2018}, whose origins trace back to \cite{dowker1952homology}.

\begin{definition}[Dowker Filtrations]\label{def:dowker_filtrations}
    Given a finite hypernetwork $H=(X,Y,\omega)\in\mathcal{FH}$ and $\delta \in \R$, we define the following relation between $X$ and $Y$:
\begin{equation}\label{eqn:relation}
    R_{\delta,H}\coloneqq\{(x,y)\mid \omega(x,y)\geq\delta\}
\end{equation}
The associated \define{Dowker $\delta$-node simplicial complex} is defined by 
\begin{equation}\label{eqn:dowkernode}
    \mathfrak{D}^{n}_{\delta,H}\coloneqq\{\sigma = [x_0,\dots,x_n]\mid \exists y'\in Y \text{ such that } (x_i,y')\in R_{\delta,H} \text{ for each }x_i\}.
\end{equation}
Similarly, the associated \define{Dowker $\delta$-edge simplicial complex} is
\begin{equation}\label{eqn:dowkeredge}
    \mathfrak{D}^{e}_{\delta,H}\coloneqq\{\sigma = [y_0,\dots,y_n]\mid \exists x'\in X \text{ such that } (x',y_i)\in R_{\delta,H} \text{ for each }y_i\} 
\end{equation} 
For $\delta \geq \delta'$, we have an inclusion $\mathfrak{D}^*_{\delta,H} \hookrightarrow  \mathfrak{D}^*_{\delta',H}$, for $*=n$ or $e$, resulting in a filtered simplicial complex. We refer to these structures, respectively, as the \define{Dowker node filtration} $\mathfrak{D}^n_H$ and \define{Dowker edge filtration} $\mathfrak{D}^e_H$ of $H$. 
\end{definition}

\begin{remark}[Comparison to Network Dowker Filtrations]\label{rem:comparison_to_network_dowker}
    In \cite{Chowdhury2018}, the Dowker filtrations of a network $N = (X,\omega)$ are defined with respect to the relation
    \[
    R_{\delta,N} = \{(x,y) \in X \times X \mid \omega(x,y) \leq \delta\}.
    \]
    Observe that the inequality is in the opposite direction of \eqref{eqn:relation}. The point is that we are interpreting the meaning of $\omega$ slightly differently here: the value of $\omega$ intuitively represents the strength or concentration of node $x$ in hyperedge $y$, whereas the $\omega$-values in the network setting of \cite{Chowdhury2018} were considered as representing distances (or dissimilarities) between nodes. This change in convention is essentially cosmetic from a mathematical perspective.  
\end{remark}

\begin{remark}[Co- Versus Contra-Variant Functors]
    The Dowker filtrations defined above can be considered as contravariant functors from the poset category $(\R,\leq)$ to the category of simplicial complexes with simplicial maps. The standard convention in TDA is to consider covariant functors between these categories. Once again (cf.\ Remark \ref{rem:comparison_to_network_dowker}), this perspective shift has no real mathematical implications, but we are nonetheless careful to check its effects in the ensuing results. 
\end{remark}

To a (contravariantly) filtered simplicial complex $\mathfrak{F} = \{\mathfrak{F}^\delta\xrightarrow{s_{\delta,\delta'}}\mathfrak{F}^{\delta'}\}_{\delta' \leq \delta}$, one typically applies degree-$k$ simplicial homology with field coefficients, which, by functoriality, gives a one-parameter family of vector spaces and linear maps. The resulting structure is called \define{degree-$k$ persistent homology} and is denoted $PH_k(\mathfrak{F})$. These structures can be compared by a natural, category-theoretic metric called \define{interleaving distance}~\cite{chazal2009proximity}, denoted $d_I(PH_k(\mathfrak{F}),PH_k(\mathfrak{G}))$. Under mild conditions (satisfied in the setting of Dowker complexes of finite hypernetworks), the interleaving distance is polynomial-time computable (because it can be reframed as a combinatorial optimization problem whose solution is known as \define{bottleneck distance}~\cite{lesnick2015theory}).

The following lemma adapts \cite[Lemma 8]{Chowdhury2018} to our contravariant setting. The proof is the same, so we omit it. It is based on the idea of \emph{contiguity}: recall that simplicial maps $f,g:\Delta \to \Sigma$ are \define{contiguous} if for every simplex $\sigma \in \Delta$, $f(\sigma) \cup g(\sigma)$ is a simplex in $\Sigma$.

\begin{lemma}[Stability Lemma]\label{lemma:contiguity}
 Let $\mathfrak{F},\mathfrak{G}$ be two filtered simplicial complexes written as
\[\{\mathfrak{F}^\delta\xrightarrow{s_{\delta,\delta'}}\mathfrak{F}^{\delta'}\} \text{   and   } \{\mathfrak{G}^\delta\xrightarrow{t_{\delta,\delta'}}\mathfrak{G}^{\delta'}\}\text{  with  }\delta' \leq \delta \in \mathbb{R},\]
where $s_{\delta,\delta'}$ and $t_{\delta,\delta'}$ are inclusion maps. Suppose $\eta\geq0$ is such that there exist families of simplicial maps $\{\varphi_\delta:\mathfrak{F}^\delta\rightarrow\mathfrak{G}^{\delta-\eta}\}_{\delta\in\mathbb{R}}$ and $\{\psi_\delta:\mathfrak{G}^\delta\rightarrow\mathfrak{F}^{\delta-\eta}\}_{\delta\in\mathbb{R}}$ such that the following are satisfied for any $\delta' \leq \delta $
\begin{enumerate}
    \item[(1)]  $t_{\delta-\eta,\delta'-\eta}\circ\varphi_\delta$ and $\varphi_{\delta'}\circ s_{\delta,\delta'}$ are contiguous.
    \item[(2)]  $s_{\delta-\eta,\delta'-\eta}\circ\psi_\delta$ and $\psi_{\delta'}\circ t_{\delta,\delta'}$ are contiguous.
    \item[(3)]  $\psi_{\delta-\eta}\circ\varphi_\delta$ and $s_{\delta,\delta-2\eta}$ are contiguous.
    \item[(4)]  $\varphi_{\delta-\eta}\circ\psi_\delta$ and $t_{\delta,\delta-2\eta}$ are contiguous.
\end{enumerate}
(See Figure \ref{fig:contiguity}.) Then, for each non-negative integer $k$,
\[ 
d_I(PH_k(\mathfrak{F}),PH_k(\mathfrak{G}))\leq\eta
\]
\end{lemma}

\begin{figure}
    \centering
    \begin{tikzcd}
\mathfrak{F}^{\delta} \arrow[rr, "{s_{\delta,\delta'}}"] \arrow[rd, "\varphi_\delta"]        &                                                                         & \mathfrak{F}^{\delta'} \arrow[rd, "\varphi_{\delta'}"] &                             &                                                                                           & \mathfrak{F}^{\delta-\eta} \arrow[rr, "{s_{\delta-\eta,\delta'-\eta}}"] &                                                    & \mathfrak{F}^{\delta'-\eta} \\
                                                                                             & \mathfrak{G}^{\delta-\eta} \arrow[rr, "{t_{\delta-\eta,\delta'-\eta}}"] &                                                        & \mathfrak{G}^{\delta'-\eta} & \mathfrak{G}^{\delta} \arrow[ru, "\psi_{\delta}"] \arrow[rr, "{t_{\delta,\delta'}}"]      &                                                                         & \mathfrak{G}^{\delta'} \arrow[ru, "\psi_{\delta}"] &                             \\
\mathfrak{F}^{\delta} \arrow[rr, "{s_{\delta,\delta-2\eta}}"] \arrow[rd, "\varphi_{\delta}"] &                                                                         & \mathfrak{F}^{\delta-2\eta}                            &                             &                                                                                           & \mathfrak{F}^{\delta-\eta} \arrow[rd, "\varphi_{\delta-\eta}"]          &                                                    &                             \\
                                                                                             & \mathfrak{G}^{\delta-\eta} \arrow[ru, "\psi_{\delta-\eta}"]             &                                                        &                             & \mathfrak{G}^{\delta} \arrow[rr, "{t_{\delta,\delta-2\eta}}"] \arrow[ru, "\psi_{\delta}"] &                                                                         & \mathfrak{G}^{\delta-2\eta}                        &                            
\end{tikzcd}
    \caption{Diagrams for the maps in Lemma \ref{lemma:contiguity}.}
    \label{fig:contiguity}
\end{figure}
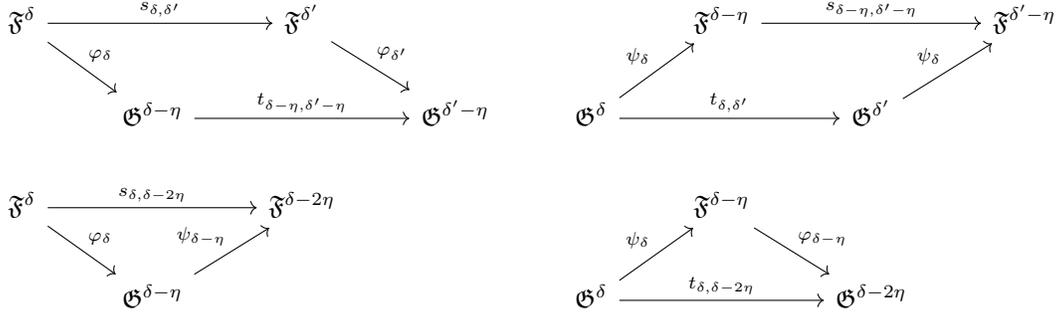

We now give the main result of this section, which gives an additional tractably computable lower bound on hypernetwork Gromov-Hausdorff distance. The proof is similar to that of \cite[Proposition 15]{Chowdhury2018}, so we only sketch it here.

\begin{theorem}[Stability of Dowker Persistent Homology]\label{thm:dowker}
For any $H,H'\in\mathcal{FH}$ and non-negative integer $k$,
\[
d_I(PH_k(\mathfrak{D}_H^*),PH_k(\mathfrak{D}_{H'}^*) )\leq d_{\mathcal{H}}(H,H'),
\]
where $\mathfrak{D}_H^* = \mathfrak{D}_H^n$ or $\mathfrak{D}_H^e$.
\end{theorem}

\begin{proof}[Proof Sketch]
Let $d_{\mathcal{H}}(H,H')=\eta$. Then, by Proposition \ref{prop:GH}, we have maps $\varphi:X\rightarrow X',\psi:Y\rightarrow Y',\varphi':X'\rightarrow X,\psi':Y'\rightarrow Y$ such that
\[ \max\{\mathrm{dis}(\varphi,\psi),\mathrm{dis}(\varphi',\psi'),\mathrm{codis}(\varphi',\psi),\mathrm{codis}(\varphi,\psi')\}\leq\eta.\]
One can then check that $\varphi,\varphi'$ induce simplicial maps $\varphi_\delta: \mathfrak{D}^n_{\delta,H}\rightarrow\mathfrak{D}^n_{\delta-\eta,H'}$ and $\varphi'_\delta: \mathfrak{D}^n_{\delta,H'}\rightarrow\mathfrak{D}^n_{\delta-\eta,H}$ for each $\delta\in\mathbb{R}$. The final step is then to check that the contiguity conditions of Lemma \ref{lemma:contiguity} hold for these maps, and this is indeed the case. This proves the statement for the node filtered simplicial complexes. In fact, the persistent homologies are the same for both the node and the edge versions, which follows from the Functorial Dowker Theorem \cite[Corollary 20]{Chowdhury2018}.
\end{proof}

Let us see how the hypernetworks from Example \ref{ex:lowerbounds} compare using their Dowker complexes.

\begin{example}\label{ex:dowker}
   Consider the hypernetworks from Figure \ref{fig:weightedomegas}. We illustrate the Dowker node complexes in Figure \ref{fig:dowker_node} and Dowker edge complexes in Figure \ref{fig:dowker_edge}. Next, we computed the \emph{barcode} for the 0-degree persistence homology (a simple, complete summary of $PH_0$---see \cite{dey2022computational}) and illustrate it in Figure \ref{fig:dowker_barcode}---recall from the proof of Theorem \ref{thm:dowker} that the persistent homologies are the same for the node- and edge-filtrations. Calculating the interleaving distance between these persistent homologies gives us a lower bound of 0.10 for $d_\mathcal{H}(H,H')$. 
\end{example}

\begin{figure}
    \centering
    \includegraphics[width=.9\columnwidth]{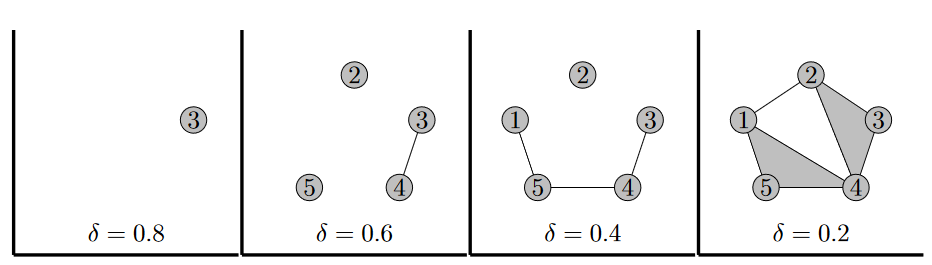}
    \includegraphics[width=.9\columnwidth]{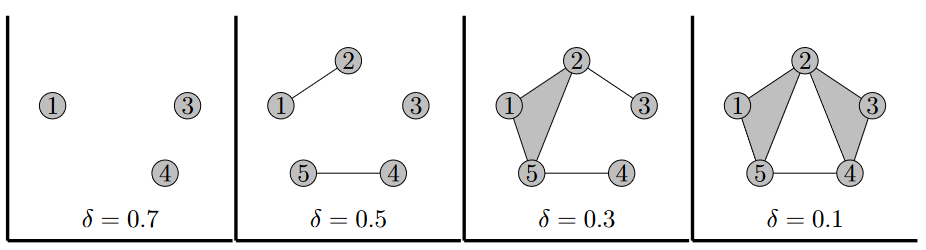}
    \caption{The Dowker node simplicial complexes for the hypernetworks $H$ (top) and $H'$ (bottom) at various $\delta$ filtrations.}
    \label{fig:dowker_node}
\end{figure}

\begin{figure}
    \centering
    \includegraphics[width=.9\columnwidth]{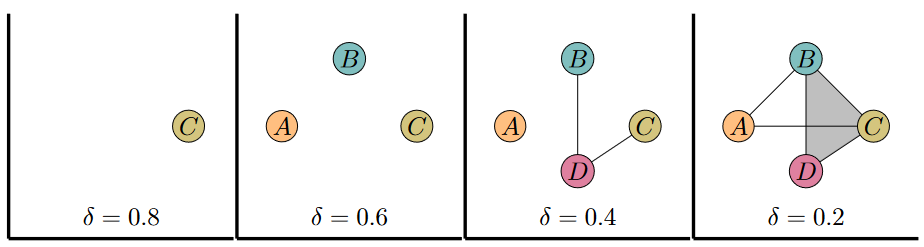}
    \includegraphics[width=.9\columnwidth]{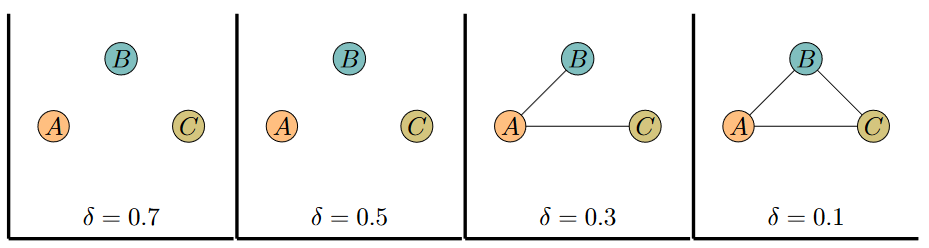}
    \caption{The Dowker edge simplicial complexes for the hypernetworks $H$ (top) and $H'$ (bottom) at various $\delta$ filtrations.}
    \label{fig:dowker_edge}
\end{figure}

\begin{figure}
    \centering
    \includegraphics[width=.9\columnwidth]{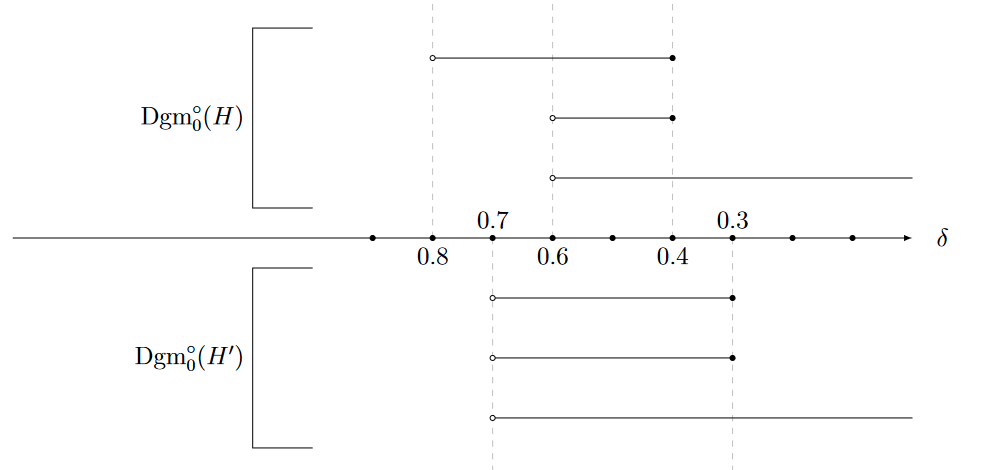}
    \caption{The 0-degree persistence barcodes for the Dowker Complexes.}
    \label{fig:dowker_barcode}
\end{figure}

\section{Stability of Cost Functions}\label{sec:stability_of_cost_functions}

The previous two sections considered lower bounds on hypernetwork Gromov-Hausdorff distance which were inspired by the perspective that hypernetworks give a generalized model for hypergraphs. In this section, we shift this perspective and consider hypernetworks as cost functions, such as those which arise in optimal transport theory (see Example \ref{ex:hypernetworks}). Due to this shift, we drop the finiteness assumption of the last two sections and work with general hypernetworks.

\subsection{The Hausdorff Map}

Let $(X,d)$ be a metric space; by abuse of notation, we denote the structure by $X$. We define the \define{Hausdorff space} for $X$ to be the metric space $(\mathrm{Haus}(X),d_\mathrm{Haus}^X)$, where $\mathrm{Haus}(X)$ denotes the space of closed and bounded subsets and $d_\mathrm{Haus}^X$ is the Hausdorff distance---recall the definition from \eqref{eqn:hausdorff_distance}. Properties of the \define{Hausdorff map} $X \mapsto \mathrm{Haus}(X)$ were studied in~\cite{mikhailov2018hausdorff} (see also applications in~\cite{memoli2023characterization}), where the main theorem shows that this mapping is 1-Lipschitz with respect to Gromov-Hausdorff distance:

\begin{theorem}[\cite{mikhailov2018hausdorff}]\label{thm:stability_of_Hausdorff_map}
    For metric spaces $(X,d)$ and $(X',d')$, the Hausdorff map satisfies
    \[
    d_{\mathcal{GH}}(\mathrm{Haus}(X),\mathrm{Haus}(X')) \leq d_{\mathcal{GH}}(X,X').
    \]
\end{theorem}

\begin{remark}
    We note that in \cite{mikhailov2018hausdorff}, the author uses the notation $\mathcal{H}(X)$ for the Hausdorff space. The letter `H' is unfortunately overloaded in the present paper, so we have opted to use $\mathrm{Haus}(X)$ instead.
\end{remark}

The Gromov-Hausdorff distance $d_\mathcal{GH}$ appearing in the theorem is the same as the one we denote $d_\mathcal{N}$, used throughout the paper. We slightly alter the notation here to emphasize that the Gromov-Hausdorff distance appearing in the theorem is specifically used in the context of comparing metric spaces, rather than the general networks we consider in this article; i.e., $d_\mathcal{GH}$ is the restriction of $d_\mathcal{N}$ to the class of metric spaces. The proof of Theorem~\ref{thm:stability_of_Hausdorff_map} in \cite{mikhailov2018hausdorff} uses the classical isometric embedding formulation, of the Gromov-Hausdorff distance---that is,
\[
d_{\mathcal{GH}}(X,X') = \inf_{Z,\iota,\iota'} d_{\mathrm{Haus}} ^Z(\iota(X),\iota'(X')),
\]
where the infimum is over metric spaces $Z$ and isometric embeddings $\iota:X \to Z$ and $\iota':X' \to Z$; this is well-known to be equivalent to the formulation of $d_{\mathcal{GH}}$ given by restricting $d_\mathcal{N}$, as defined in Definition \ref{def:network_GH}.

Below, we provide an alternative proof of Theorem \ref{thm:stability_of_Hausdorff_map}, which uses the mapping formulation of Gromov-Hausdorff distance (as described at the beginning of Section \ref{sec:mapping_formulation}), thereby allowing us to generalize the theorem to the setting of networks. This new proof strategy will allow us to extend it to the space of hypernetworks. To begin, we extend the definition of the Hausdorff map: given a (not necessarily finite) network $N = (X,\omega)$, we define the network $\mathrm{Haus}(N) = (\mathrm{Haus}(X),\omega_\mathrm{Haus})$, where $\mathrm{Haus}(X)$ is the space of subsets of $X$ which are bounded with respect to $\omega$,
\[
\mathrm{Haus}(X) \coloneqq \left\{A \subset X \mid \sup_{x,y \in A} \omega(x,y) < \infty\right\},
\]
and 
\[
\omega_\mathrm{Haus}(A,B) \coloneqq \max\left\{\sup_{x \in A} \inf_{y \in B} \omega(x,y), \sup_{y \in B} \inf_{x \in A} \omega(x,y) \right\}.
\]
The map $N \mapsto \mathrm{Haus}(N)$ is called the \define{network Hausdorff map}.

\begin{theorem}\label{thm:network_hausdorff_map}
    For networks $N$ and $N'$, the network Haudorff map satisfies
    \[
    d_{\mathcal{N}}(\mathrm{Haus}(N),\mathrm{Haus}(N')) \leq d_{\mathcal{N}}(N,N').
    \]
\end{theorem}

For convenience, we introduce the notation
\[
\vec{\omega}_\mathrm{Haus}(A , B) \coloneqq \sup_{x \in A} \inf_{y \in B} \omega(x,y) \quad \mbox{and} \quad \cev{\omega}_\mathrm{Haus}(A , B) \coloneqq \sup_{y \in B} \inf_{x \in A} \omega(x,y),
\]
so that 
\[
\omega_\mathrm{Haus} = \max \left\{\vec{\omega}_\mathrm{Haus}, \cev{\omega}_\mathrm{Haus} \right\}.
\]

\begin{proof}
    Let $\varphi:X \to X'$ and $\psi:X' \to X$ be maps satisfying
    \[
    \mathrm{dis}(\varphi), \, \mathrm{dis}(\psi), \, \mathrm{codis}(\varphi,\psi) \leq \epsilon.
    \]
    We extend these to maps $\varphi_\mathrm{Haus}:\mathrm{Haus}(X) \to \mathrm{Haus}(X')$ and $\psi_\mathrm{Haus}:\mathrm{Haus}(X') \to \mathrm{Haus}(X)$ in the obvious way, i.e.,
    \[
    \varphi_\mathrm{Haus}(A) = \varphi(A),
    \]
    where the latter means the image of the set $A \in \mathrm{Haus}(X)$ under $\varphi$. Our goal is to show that the distortions of these induced maps are also bounded by $\epsilon$. 

    Let us first consider $\mathrm{dis}(\varphi_\mathrm{Haus})$. Fix $A,B \in \mathrm{Haus}(X)$ and let $\eta > 0$ be arbitrary. We have 
    \begin{align*}
    &\omega_\mathrm{Haus}(A,B) - \omega_\mathrm{Haus}'(\varphi_\mathrm{Haus}(A),\varphi_\mathrm{Haus}(B)) \\
    &= \max\left\{\vec{\omega}_\mathrm{Haus}(A,B), \cev{\omega}_\mathrm{Haus}(A,B)\right\} - \max\left\{\vec{\omega}_\mathrm{Haus}'(\varphi_\mathrm{Haus}(A),\varphi_\mathrm{Haus}(B)), \cev{\omega}_\mathrm{Haus}'(\varphi_\mathrm{Haus}(A),\varphi_\mathrm{Haus}(B))\right\}.
    \end{align*}
    Without loss of generality, suppose that $\omega_\mathrm{Haus}(A,B) = \vec{\omega}_\mathrm{Haus}(A,B)$. Then the above is bounded by 
    \[
    \vec{\omega}_\mathrm{Haus}(A,B) - \vec{\omega}_\mathrm{Haus}'(\varphi_\mathrm{Haus}(A),\varphi_\mathrm{Haus}(B)) = \sup_{x \in A} \inf_{y \in B} \omega(x,y) - \sup_{x' \in \varphi(A)} \inf_{y' \in \varphi(B)} \omega'(x',y').
    \]
    Choosing $x_0 \in A$ such that 
    \[
    \sup_{x \in A} \inf_{y \in B} \omega(x,y) \leq \inf_{y \in B} \omega(x_0,y) + \eta,
    \]
    we have 
    \[
    \sup_{x \in A} \inf_{y \in B} \omega(x,y) - \sup_{x' \in \varphi(A)} \inf_{y' \in \varphi(B)} \omega'(x',y') \leq \inf_{y \in B} \omega(x_0,y) - \inf_{y' \in \varphi(B)} \omega'(\varphi(x_0),y') + \eta.
    \]
    Next, choose $\varphi(y_0) \in \varphi(B)$ such that
    \[
    \inf_{y' \in \varphi(B)} \omega'(\varphi(x_0),y') \geq \omega'(\varphi(x_0),\varphi(y_0)) - \eta,
    \]
    so that 
    \[
    \inf_{y \in B} \omega(x_0,y) - \inf_{y' \in \varphi(B)} \omega'(\varphi(x_0),y') + \eta \leq \omega(x_0,y_0) - \omega'(\varphi(x_0),\varphi(y_0)) + 2 \eta \leq \epsilon + 2\eta,
    \]
    where we have used the assumption that $\mathrm{dis}(\varphi) \leq \epsilon$. Since this holds for all $\eta > 0$, we have shown that 
    \[
    \omega_\mathrm{Haus}(A,B) - \omega_\mathrm{Haus}'(\varphi_\mathrm{Haus}(A),\varphi_\mathrm{Haus}(B)) \leq \epsilon.
    \]
    A similar argument shows that 
    \[
    -\omega_\mathrm{Haus}(A,B) + \omega_\mathrm{Haus}'(\varphi_\mathrm{Haus}(A),\varphi_\mathrm{Haus}(B)) \leq \epsilon,
    \]
    so that 
    \[
    |\omega_\mathrm{Haus}(A,B) - \omega_\mathrm{Haus}'(\varphi_\mathrm{Haus}(A),\varphi_\mathrm{Haus}(B))| \leq \epsilon 
    \]
    holds for all $A,B \in \mathrm{Haus}(X)$. Therefore $\mathrm{dis}(\varphi_\mathrm{Haus}) \leq \epsilon$. 

    It is not hard to see that this argument can be adapted to show that both $\mathrm{dis}(\psi_\mathrm{Haus})$ and $\mathrm{codis}(\varphi_\mathrm{Haus},\psi_\mathrm{Haus})$ are also bounded above by $\epsilon$. This completes the proof.
\end{proof}

Finally, we extend this story to the setting of hypernetworks. Let $H = (X,Y,\omega)$ and $H' = (X',Y',\omega')$ be hypernetworks---we do not assume that the hypernetworks are finite, but we impose the mild constraint that the hypernetwork functions $\omega$ and $\omega'$ are bounded. We define the \define{hypernetwork Hausdorff map} in the natural way: $\mathrm{Haus}(H) = (\mathrm{Haus}(X),\mathrm{Haus}(Y),\omega_\mathrm{Haus}$), where $\mathrm{Haus}(X)$ and $\mathrm{Haus}(Y)$ are the collections of all subsets of $X$ and $Y$, respectively, and 
\[
\omega_\mathrm{Haus}(A,B) \coloneqq \max\left\{\vec{\omega}_\mathrm{Haus}(A,B), \cev{\omega}_\mathrm{Haus}(A,B) \right\} \coloneqq \max\left\{\sup_{x \in A} \inf_{y \in B} \omega(x,y), \sup_{y \in B} \inf_{x \in A} \omega(x,y) \right\}.
\]

\begin{theorem}\label{thm:hypernetwork_hausdorff_map}
    For hypernetworks $H$ and $H'$, the hypernetwork Hausdorff map satisfies
    \[
    d_{\mathcal{H}}(\mathrm{Haus}(H),\mathrm{Haus}(H')) \leq d_{\mathcal{H}}(H,H').
    \]
\end{theorem}

Since the hypernetwork distance $d_\mathcal{H}$ can be expressed in terms of distortions of mappings (Proposition \ref{prop:GH}), the proof of Theorem \ref{thm:hypernetwork_hausdorff_map} can be obtained by only superficially adapting the proof strategy of Theorem \ref{thm:network_hausdorff_map}. We omit the details here. 

\begin{remark}[Intuitive Interpretation]\label{rem:intuitive_interpretation_hausdorff}
    For a hypernetwork $H = (X,Y,\omega)$ considered as an abstract cost function, we can consider the Hausdorff space $(\mathrm{Haus}(X),\mathrm{Haus}(Y),\omega_\mathrm{Haus})$ as a proxy for the $p=\infty$ Wasserstein space from optimal transport theory. Intuitively, Theorem \ref{thm:hypernetwork_hausdorff_map} says that if a pair of cost functions is close (with respect to $d_\mathcal{H}$) then the resulting (proxies for) Wasserstein spaces are also close. 
\end{remark}

\subsection{Non-Negative Cross Curvature}

In the recent work~\cite{leger2024nonnegative}, the authors introduce a synthetic version of the famous Ma-Trudinger-Wang condition from optimal transport theory~\cite{ma2005regularity}. Comparing with the formalism used therein, it is immediately clear that it fits well within the framework presented in this article. We recall \cite[Definition 1.1]{leger2024nonnegative}, stated in the terminology of the present paper. 

\begin{definition}\label{def:NNCC}
    Let $H = (X,Y,\omega)$ be a hypernetwork which has a bounded network function. We say that $H$ is a \define{non-negative cross curvature (NNCC) space} if for every $(x_0,x_1,\bar{y}) \in X \times X \times Y$ there exists a path $x:[0,1] \to X$ such that $x(0) = x_0$, $x(1) = x_1$ and for all $y \in Y$ and $s \in [0,1]$, 
    \[
    \omega(x(s),\bar{y}) - \omega(x(s),y) \leq (1-s) \big(\omega(x_0,\bar{y}) - \omega(x_0,y) \big) + s \big(\omega(x_1,\bar{y}) - \omega(x_1,y) \big).
    \]
\end{definition}

\begin{remark}
    The cost functions used in~\cite{leger2024nonnegative} are allowed to take the values $\pm \infty$, so that the above definition is less general, but avoids some technical conditions. 
\end{remark}

The definition of an NNCC space is reminiscent of that of a non-negatively curved Alexandrov space; indeed, \cite[Proposition 1.3]{leger2024nonnegative} shows that if $(X,d)$ is a metric space and the associated hypernetwork $(X,X,d^2)$ is an NNCC space, then $(X,d)$ is non-negatively curved in the sense of Alexandrov. A main result of \cite{leger2024nonnegative} is that the property of being an NNCC space is preserved under Gromov-Hausdorff limits, when the underlying hypernetwork is derived from a metric space. This suggests the following natural extension:

\begin{theorem}\label{thm:nonnegative_cross_curvature}
    Let $H^n = (X^n,Y^n,\omega^n)$, $n=1,2,\ldots$, be a sequence of hypernetworks that converges to a hypernetwork $H = (X,Y,\omega)$, with respect to the hypernetwork distance $d_\mathcal{H}$. Suppose that $X$ is compact and metrizable by a metric $d$, such that for all $y \in Y$, the function $\omega(\cdot,y):X \to \R$ is locally Lipschitz with respect to $d$. If each $H^n$ is an NNCC space, then so is $H$.
\end{theorem}

The proof of the theorem follows the main ideas of \cite[Theorem 1.4]{leger2024nonnegative}. The main contribution of this result is the observation that the stability of NNCC extends to a much larger class of spaces when framed in terms of the hypernetwork distance. 

\begin{proof}
    Let $\epsilon^n$ be a sequence of positive numbers converging to zero as $n \to \infty$. For each $n$, there exists correspondences $S^n \in \mathcal{R}(X^n,X)$ and $T^n \in \mathcal{R}(Y^n,Y)$ such that $\mathrm{dis}_\mathcal{H}(S^n,T^n) \leq \epsilon^n$. Fix $(x_0,x_1,\bar{y}) \in X \times X \times Y$. For each $n$, choose $(x_0^n,x_1^n,\bar{y}^n) \in X^n \times X^n \times Y^n$ with $(x_0^n,x_0)$, $(x_1^n,x_1) \in R^n$ and $(\bar{y}^n,\bar{y}) \in S^n$, and let $x^n(s)$ be a path satisfying the conditions of Definition \ref{def:NNCC}. For each $s \in [0,1]$ and each $n$, choose $z^n(s)$ such that $(x^n(s),z^n(s)) \in R^n$. By compactness of $X$, we can assume without loss of generality (passing to a subsequence, if necessary) that $z^n(s)$ converges to some point of $X$ as $n \to \infty$; we denote this point as $x(s)$, so that $s \mapsto x(s)$ defines a path in $X$. 

    It remains to show that the path $x$ satisfies the defining condition for an NNCC space. For sufficiently large $n$, we have the following string of inequalities, which are justified below
    \begin{align}
        &\omega(x(s),\bar{y}) - \omega(x(s),y) \nonumber \\
        &\leq \omega(z^n(s),\bar{y}) - \omega(z^n(s),y) + 2K\cdot d(x(s),z^n(s)) \label{eqn:NNCC_proof_1} \\
        &\leq \omega^n(x^n(s),\bar{y}^n) - \omega^n(x^n(s),y^n) + 2\epsilon^n + 2K\cdot d(x(s),z^n(s)) \label{eqn:NNCC_proof_2} \\
        &\leq (1-s)\big(\omega^n(x^n_0,\bar{y}^n) - \omega^n(x^n_0,y^n)\big) + s\big(\omega^n(x^n_1,\bar{y}^n) - \omega^n(x^n_1,y^n)\big)+ 2\epsilon^n + 2K\cdot d(x(s),z^n(s)) \label{eqn:NNCC_proof_3}\\
        &\leq (1-s)\big(\omega(x_0,\bar{y}) - \omega(x_0,y)\big) + s\big(\omega(x_1,\bar{y}) - \omega(x_1,y)\big)+ 4\epsilon^n + 2K\cdot d(x(s),z^n(s)). \label{eqn:NNCC_proof_4}
    \end{align}
    The constant $K$ appearing in \eqref{eqn:NNCC_proof_1} is some upper bound on the local Lipschitz constants for the functions $\omega(\cdot,\bar{y})$ and $\omega(\cdot,y)$; then the inequality in \eqref{eqn:NNCC_proof_1} follows because
    \[
    \omega(x(s),\bar{y}) - \omega(z^n(s),\bar{y}) \leq K\cdot d(x(s),z^n(s)) \quad \mbox{and} \quad \omega(z^n(s),y) - \omega(x(s),y) \leq K\cdot d(x(s),z^n(s)),
    \]
    for sufficiently large $n$. The bound \eqref{eqn:NNCC_proof_2} follows because $(x^n(s),z^n(s)) \in R^n$ and $(\bar{y}^n,\bar{y}) \in S^n$, so that 
    \[
    \omega(z^n(s),\bar{y}) - \omega^n(x^n(s),\bar{y}^n) \leq \epsilon^n,
    \]
    with a similar bound applying to the other term. The assumption that each $H^n$ is an NNCC space gives the bound \eqref{eqn:NNCC_proof_3}. Finally, arguments similar to those used to justify \eqref{eqn:NNCC_proof_2} imply the last bound \eqref{eqn:NNCC_proof_4}. Taking the limit $n \to \infty$ yields the inequality required to show that $H$ is an NNCC space.
\end{proof}

\bibliographystyle{plain}
\bibliography{sample}

\end{document}